%% file: VitZoll.tex
\documentclass[11pt]{amsart}

\usepackage{eucal}
\usepackage{upgreek}
\usepackage{pdfsync}
\usepackage[all, cmtip]{xy}
\usepackage{amsfonts}
\usepackage{amsmath}
\usepackage{amssymb}
\usepackage{amscd}
\usepackage{color,xcolor}
\usepackage{amsthm}
\usepackage{amsxtra}
\usepackage{bbm}
\usepackage{graphicx}
\usepackage[hmargin=3cm,vmargin=3cm]{geometry}
\usepackage{mathrsfs}
\usepackage[numbers,sort]{natbib}
\usepackage{isomath}
\usepackage{enumitem}
\usepackage{latexsym} 

\newtheorem{thm}{Theorem}

\newtheorem{prop}{Proposition}
\newtheorem{lma}[prop]{Lemma}

\newtheorem{cor}[prop]{Corollary}

\theoremstyle{definition}

\newtheorem{df}[prop]{Definition} 

\newtheorem{conj}[prop]{Conjecture}

\theoremstyle{remark}

\newtheorem{rmk}[prop]{Remark} 

\def\mrm#1{{\mathrm{#1}}}
\def\bb#1{{\mathbb{#1}}}
\def\cl#1{{\mathcal{#1}}}
\def\ul#1{{\underline{#1}}}

\newcommand{\R}{{\mathbb{R}}}
\newcommand{\Z}{{\mathbb{Z}}}
\newcommand{\C}{{\mathbb{C}}}

\newcommand{\D}{{\mathbb{D}}}
\newcommand{\bK}{{\mathbb{K}}}

\newcommand{\bH}{{\mathbb{H}}}
\newcommand{\bF}{{\mathbb{F}}}
\newcommand{\pt}{{[pt]}}

\newcommand{\rH}{{\mathrm{H}}}

\newcommand{\del}{\partial}

\newcommand{\sm}[1]{C^\infty(#1)}

\newcommand{\cL}{\mathcal{L}}

\newcommand{\til}[1]{\widetilde{#1}}

\newcommand{\bra}[1]{{\{ #1 \}}}
\newcommand{\codim}{\text{codim}}

\newcommand{\om}{\omega}
\newcommand{\sig}{{\sigma}}

\newcommand{\la}{\lambda}

\newcommand{\eps}{\epsilon}

\newcommand{\cA}{\mathcal{A}}
\newcommand{\cB}{\mathcal{B}}
\newcommand{\cC}{\mathcal{C}}
\newcommand{\cD}{\mathcal{D}}
\newcommand{\cE}{\mathcal{E}}

\newcommand{\cH}{\mathcal{H}}

\newcommand{\cJ}{\mathcal{J}}

\newcommand{\cO}{\mathcal{O}}

\newcommand{\cP}{\mathcal{P}}

\newcommand{\LH}{H,L}

\newcommand{\tmin}{{\text{min},\bK}}
\newcommand{\tmon}{{\text{mon},\bK}}
\newcommand{\tuniv}{{\text{univ},\bK}}

\newcommand{\CP}{{\C P^n, \om_{FS}}}

\renewcommand{\bar}[1]{\overline{#1}}
\renewcommand{\hat}[1]{\widehat{#1}}

\DeclareMathOperator{\id}{\mathrm{id}}

\DeclareMathOperator{\ind}{\mathrm{ind}}
\DeclareMathOperator{\supp}{\mathrm{supp}}
\DeclareMathOperator{\rank}{\mathrm{rank}}

\DeclareMathOperator{\Ham}{\mathrm{Ham}}

\DeclareMathOperator{\Symp}{\mathrm{Symp}}

\let\Im\relax
\DeclareMathOperator{\Im}{\mathrm{Im}}

\DeclareMathOperator{\ima}{\mathrm{im}}

\DeclareMathOperator{\Spec}{\mathrm{Spec}}

\DeclareMathOperator{\pemod}{{\mathbf{pmod}}}
\DeclareMathOperator{\barc}{{\mathbf{barcodes}}}

\DeclareMathOperator{\diag}{\mathrm{diag}}

\def\H2{H^{(2)}}

\usepackage[hyperindex]{hyperref}

\newcommand{\esemail}{shelukhin@dms.umontreal.ca}

\begin{document}
	

\title{Viterbo conjecture for Zoll symmetric spaces}


\author{Egor Shelukhin}
\address{Egor Shelukhin, Department of Mathematics and Statistics,
	University of Montreal, C.P. 6128 Succ.  Centre-Ville Montreal, QC
	H3C 3J7, Canada}
\email{\esemail}

\bibliographystyle{abbrv}

\begin{abstract}
	
	We prove a conjecture of Viterbo from 2007 on the existence of a uniform bound on the Lagrangian spectral norm of Hamiltonian deformations of the zero section in unit cotangent disk bundles, for bases given by compact rank one symmetric spaces $S^n, \R P^n, \C P^n, \bH P^n,$ $n\geq 1.$ We discuss generalizations and give applications, in particular to $C^0$ symplectic topology. Our key methods, which are of independent interest, consist of a reinterpretation of the spectral norm via the asymptotic behavior of a family of cones of filtered morphisms, and a quantitative deformation argument for Floer persistence modules, that allows to excise a divisor.
	
	
	
	
	
	
\end{abstract}

\maketitle

\tableofcontents

\input{intro.tex}
\input{prelimVZ.tex}






 \input{Viterbo4.tex}

\input{applications.tex}

\section{Discussion}
 
We discuss the case of $L$ given by a product of Lagrangians $L_i \in \cl V,$ with $N_{L_i} = N$ for all $1 \leq i \leq m,$ so that $L \subset M,$ is monotone, where $M$ is given by the product of the respective $\cl M_i \in \cl W.$ The space $L = T^n$ falls into this class of examples. The methods described in this paper yield the following upper bound on the spectral norm. This is the best result we were able to obtain, while staying in the setting of monotone Lagrangian submanifolds. Investigation of the non-monotone case shall be carried out elsewhere. 

\begin{prop}
Let $L' \subset U = D^*L_1 \times \ldots D^* L_m$ be Hamiltonian isotopic to $L = L_1 \times \ldots \times L_m,$   inside $M = M_1 \times \ldots \times M_m.$ Then for $c_j = c_{L_j},$ \[\gamma(L,L';\bK, U) \leq \frac{\sum_{j=1}^{m} c_j}{1-\max_j c_j} \cdot (A_L + \beta(L,L';\bK,U)).\] 
\end{prop}

Since providing a uniform bound on $\beta(L,L';\bK,U)$ seems to be essentially as difficult as doing so for $\gamma(L,L';\bK,U),$ we shall not present the proof of this statement in detail, merely remarking that one should use the cones $Cone_{\sigma}(a,H,L;\cD)$ for quantum roots of unity $a$ given by $a = a_1 \otimes \ldots \otimes a_m,$ with $a_i = [L]$ for all $i \neq j,$ and $a_j = [pt].$

\bibliography{bibliographyVZ}

\end{document}

%% file: intro.tex
\section{Introduction and main results}

In 2007 Viterbo (see \cite[Conjecture 1]{Viterbo-homog}) has conjectured that the spectral norm $\gamma(L,L')$ \cite{Viterbo-specGF} of each exact Lagrangian deformation $L'$ of the zero section in the unit co-disk bundle $D_g^*L \subset T^*L$ of the closed manifold $L = T^n,$ with respect to a Riemannian metric $g,$ is uniformly bounded by a constant $C(g,L).$ The spectral norm is given by the difference of two homological minimax values in suitable generating function homology, and can be recast in terms of Lagrangian Floer homology \cite{MonznerVicheryZapolsky}. This conjecture has since been completely open.

In this paper we start with the observation that the conjecture of Viterbo makes sense, and is of interest for arbitrary closed manifolds $L.$ Our main theorem proves it for all $L$ belonging to the four infinite families $\cl V = \{\R P^n, \C P^n, \bH P^n, S^n\}_{n\geq 1}$ of compact rank one symmetric spaces. In particular we prove the original conjecture of Viterbo for $n=1.$ We observe the curious fact that while it is known that $\gamma(L,L') \leq d_{\mrm{Hofer}}(L,L')$ for the Lagrangian Hofer metric \cite{ChekanovFinsler}, even in the case of $L=S^1,$ it is easy to see by an argument of Khanevsky \cite{Kha-diam} that $d_{\mrm{Hofer}}$ has infinite diameter. 


We remark that all homological notions and computations in this paper depend a priori on the choice of coefficients. We work with the ground field $\bK = \bF_2$ throughout the paper, and write $\gamma(L,L'),\gamma(L,L';\bK),$ or $\gamma(L,L'; W)$ if we want to emphasize the symplectic manifold $W$ wherein it is computed, for the corresponding spectral norm.

\begin{thm}\label{thm:Viterbo}
Let $\bK = \bF_2,$ and $L \in \cl V,$ equipped with a Riemannian metric $g.$ Then there exists a constant $C(g,L)$ such that \[\gamma(L,L';\bK) \leq C(g,L)\] for all exact Lagrangian deformations $L'$ of the zero section $L$ in $D^*_g L.$
\end{thm}

%

We note that since $\gamma(L,L';\bK)$ bounds from above the difference between each two spectral invariants of $(L,L'),$ the same bound applies to each such difference. To slightly strengthen this result, and to reflect on its proof, we recall that manifolds in $\cl V$ are precisely those compact symmetric spaces that admit a Riemannian metric with all prime geodesics closed and of the same length\footnote{Save for the isolated case of the Cayley plane $\bb OP^2,$ that will be treated separately elsewhere.}. We normalize this length to be equal to $2,$ so that the metric have diameter $1,$ and equip each manifold $L \in \cl V$ with the resulting standard Zoll Riemannian metric \cite{BesseClosed}, unless otherwise stated.

The Zoll cut construction \cite{Audin-Zoll}, which is a special case of the symplectic cut construction \cite{Lerman-cuts}, allows us to embed each manifold in $L \in \cl V$ as a monotone Lagrangian submanifold of a closed monotone symplectic manifold $M,$ in such a way as to exhibit the open unit cotangent disk bundle $D^*L$ as the complement $M \setminus \Sigma$ of a symplectic Donaldson divisor $\Sigma \subset M.$ The manifolds $M,$ up to scaling of symplectic forms, belong to the four infinite families $\cl W = \{\C P^n, \C P^n \times (\C P^n)^{-}, Gr(2,2n+2), Q^n\}_{n \geq 1},$ respectively, where $\C P^n$ is endowed with the standard Fubini-Study symplectic form $\om_{FS},$ $(\C P^n)^{-}$ denotes $(\C P^n, - \om_{FS}),$ $Gr(2,2n+2)$ is the complex Grassmannian of two-planes in $\C^{2n+2},$ and $Q^n$ is a smooth complex quadric in $\C P^{n+1}.$ Record the dimensions $n_L = n, 2n, 4n, n$ of $L \in \cl V,$ the minimal Maslov numbers $N_L = n+1, 2n+2, 4n+4, 2n$ of $L \in \cl V$ as a Lagrangian submanifold of $M \in \cl W,$ and set \[c_L = \frac{n_L}{N_L}<1.\] The normalization of the symplectic form on $M \in \cl W$ that is naturally obtained from the Zoll cut construction is such that the minimal symplectic area of a disk in $M$ with boundary on $L$ is $A_L = \frac{1}{2}.$

\begin{thm}\label{thm:Viterbo-sharper}
Let $L \in \cl V$ be a Lagrangian submanifold of $M \in \cl W.$ Let $L' \subset D^*L$ be an exact Lagrangian submanifold of $D^*L$ that is an exact Lagrangian deformation of the zero section $L,$ considered as a Lagrangian submanifold of $M.$ That is, there exists a Hamiltonian isotopy $\{\phi^t\}$ of $M$ with $\phi^1(L) = L' \subset D^*L \subset M.$ Then for $c=c_L$ \begin{equation}\label{eq: gamma bound Viterbo}\gamma(L,L';\bK) \leq \frac{(1+c)c}{2(1-c)}.\end{equation}
\end{thm}

One key topological property of Zoll manifolds is that their homology algebra is isomorphic to a truncated polynomial ring $H_*(L;\bK) \cong \bK[a]/(a^{n+1}),$ on a homogeneous element $a.$ The class $a$ is given by $[\R P^{n-1}], [\bb CP^{n-1}], [\bb HP^{n-1}],[pt]$ in the case of $L = \R P^n, \C P^n, \bH P^n, S^n,$ respectively. This enables one to prove \cite{BiranCorneaRigidityUniruling,SmithPencils,KS-bounds} that for $L \in \cl V,$ the self-Floer homology $HF(L) \cong QH(L)$ of $L,$ also known as the Lagrangian quantum homology \cite{BiranCorneaLagrangianQuantumHomology}, with coefficients in the Novikov field $\Lambda_{L,\tmin} = \bK[q^{-1},q]]$ with quantum variable $q$ of degree $(-N_L)$ satisfies, as a ring \[QH(L) \cong \Lambda_{L,\tmin}[a]/(a^{n+1} = q).\] In patricular $\dim QH_r(L) \leq 1,$ for all $r \in \Z,$ and $[pt] = a^n$ is a quantum root of unity: $[pt]^{n+1} = q^n.$ This allows us to provide a uniform bound \cite{KS-bounds} on the spectral norm and boundary depth of $(L,L')$ as computed in $M,$ that is {\em smaller than the minimal area of a pseudo-holomorphic disk} in $M$ with boundary on $L$. In fact, there is an algebraic version of the spectral norm, that is easy to see to be sufficient for our purposes, and to require only the above algebraic properties. Reinterpreting the spectral norm in terms of cones of filtered morphisms depending on a large parameter, and proving an algebraic deformation argument for the suitable persistence modules allows us to deduce from the above upper bound our desired result, by means of symplectic field theory \cite{BEHWZ-compactness}.

Of course Theorem \ref{thm:Viterbo} is a direct consequence of Theorem \ref{thm:Viterbo-sharper}. However, Theorem \ref{thm:Viterbo-sharper} has additional applications, and moreover suggests the following generalized Viterbo conjecture.

\begin{conj}\label{conj:Viterbo-generalized}
Let $\bK = \bF_2,$ and $L$ a closed manifold equipped with a Riemannian metric $g.$ Then there exists a constant $C(g,L)$ such that \[\gamma(L,L';\bK) \leq C(g,L)\] for all exact Lagrangian submanifolds $L' \subset D^*_g L.$
\end{conj}

\begin{rmk}
This conjecture and hence Theorem \ref{thm:Viterbo-sharper} would follow from Theorem \ref{thm:Viterbo} for $L \in \cl V$, if the nearby Lagrangian conjecture (see \cite{AK-simplehomotopy}) were true. In particular it holds for $L = S^1$ by a folklore argument, for $L=\C P^1$ by a combination of \cite{Hind-nearby} and \cite{Ritter-thesis}, and for $L=\R P^2$ by \cite{Hind-Stein}. We expect it to be possible to prove this generalized conjecture for all $L \in \cl V$ by verifying more algebraically the conditions of Proposition \ref{prop: deformation} for all exact $L' \subset D^*M \subset M.$ In the generality of Conjecture \ref{conj:Viterbo-generalized}, the best known result is currently an upper bound of $\gamma(L,L';\bK)$ that is linear in the boundary depth \cite{UsherBD1,UsherBD2} $\beta(L',F;\bK)$ of the Floer complex of $L'$ with a Lagrangian fiber $F$ of $T^* L$ \cite{BC-private}. Finally, while in this paper we work with coefficients in $\bK = \bF_2,$ we expect the same statement for $\{\C P^n, \bb H P^n, S^n\}$ to hold with arbitrary choice of ground field $\bK.$
\end{rmk}


\subsection{Applications}

\subsubsection{$C^0$ symplectic topology}

The first application of Theorem \ref{thm:Viterbo-sharper} is the following $C^0$-continuity statement for the Hamiltonian spectral norm, inspired by \cite[Remark 1.9]{SeyfaddiniC0Limits}. Considering the distance function $d$ on $M = \C P^n$ coming from a Riemannian metric, we define the following distance function on the group $\Ham(M,\om):$ for $\phi,\phi' \in \Ham(M,\om),$ set \[d_{C^0}(\phi,\phi') = \min_{x \in M} d(\phi(x),\phi'(x)).\] We call the topology induced by $d_{C^0}$ the $C^0$ topology on $\Ham(M,\om).$ Recall that in \cite{Oh-specnorm} following \cite{Schwarz:action-spectrum,Viterbo-specGF}, a spectral norm $\gamma:\Ham(M,\om) \to \R_{\geq 0}$ on any closed symplectic manifold was introduced and shown to be non-degenerate. Moreover, $\gamma$ provides a lower bound on the celebrated Hofer norm \cite{HoferMetric,Lalonde-McDuff-Energy}. 

\begin{thm}\label{thm:gamma is C^0 continuous on CP^n}
	The spectral norm $\gamma: \Ham(\C P^n, \om_{FS}) \to \R$ is continuous with respect to the $C^0$-topology on $Ham(\C P^n, \om_{FS}).$ In fact, when $d$ comes from the Zoll metric, we obtain for all  $\phi \in \Ham(\bb CP^n,\om_{FS})$ the inequality \begin{equation}\label{eq:gamma C0 bound lin}\gamma(\phi) < C_n \cdot d(\phi,\id),\end{equation} for the constant  $C_n = \frac{(1+c)c}{1-c} = \frac{n(2n+1)}{n+1}$ for $c = c_{\C P^n} = \frac{n}{n+1}.$
\end{thm}


\begin{rmk}
This statement implies that the spectral norm $\gamma$ is continuous in the $C^0$ topology. This latter fact was known for $(M,\om)$ being a closed symplectic surface \cite{SeyfaddiniC0Limits}. During the preparation of this paper, this was also shown in \cite{BHS-spectrum} for closed symplectically aspherical manifolds $(M,\om).$ In the case $(M,\om) = (\C P^1, \om_{FS})$ the linear bound \eqref{eq:gamma C0 bound lin} in $d(\phi,\id)$ improves upon the H\"{o}lder bound of exponent $\frac{1}{2}$ in \cite{SeyfaddiniC0Limits}.
\end{rmk}

Similarly to the $C^0$-continuity result of \cite{BHS-spectrum}, Theorem \ref{thm:gamma is C^0 continuous on CP^n} has further applications in $C^0$ symplectic topology, extending results that were previously known for the most part in dimension $2,$ or for open symplectic manifolds, to closed higher-dimensional symplectic manifolds. First, a partial answer to a question of Le Roux \cite{Leroux-six} for $(\CP)$ follows. Statements of this kind first appeared in \cite{EPP-C0} for $\D^{2n},$ and for a class of closed aspherical symplectic manifolds containing $T^{2n},$ and in \cite{Seyfaddini-descent} for certain additional open symplectic manifolds. 

\begin{cor}
 Let $E_A = \{\phi \in \Ham(\CP)\,|\, d_{\mrm{Hofer}}(\phi,\id) > A\}.$ Then for all $A \in [0,\frac{n}{n+1})$ the interior of $E_A$ in $(\Ham(\CP),d_{C^0})$ is non-empty.
\end{cor}

\begin{proof}
 For all $A \in [0,\frac{n}{n+1}),$ by \cite[Theorem F]{KS-bounds} there exists $\phi \in \Ham(\C P^n, \om_{FS})$ such that $\gamma(\phi) > A,$ and hence $d_{\mrm{Hofer}}(\phi,\id) > A.$ Moreover, by $C^0$ continuity of $\gamma,$ there is an open $C^0$-ball $B$ around $\phi$ in $\Ham(\C P^n, \om_{FS})$ with $\gamma|_B > A,$ and hence $B \subset E_A.$
\end{proof}


Second, the displaced disks problem of B\'{e}guin, Crovisier, and Le Roux, solved for closed surfaces in \cite{Seyfaddini-disk}, follows for $(\CP).$

\begin{cor}
Let $\phi \in \overline{\Ham}(\CP)$ be a Hamiltonian homeomorphism that displaces a symplectic closed ball $B$ of radius $r.$ Then $d_{C^0}(\phi,\id) \geq \frac{\pi r^2}{C_n} > 0.$ 
\end{cor}

\begin{proof}
If $\phi_k \in \Ham(\CP)$ satisfies $d_{C^0}(\phi_k,\phi) \xrightarrow{n \to \infty} 0,$ then $\phi_k$ displaces $B$ for all $k \gg 1.$ Hence by \cite[Remark 2.2]{Usher-sharp} $\gamma(\phi_k) \geq \pi r^2$ for all $n \gg 1.$ In particular $\pi r^2 < \frac{1}{2},$ by Gromov's $2$-ball theorem \cite{GromovPseudohol}. Hence by Theorem \ref{thm:gamma is C^0 continuous on CP^n} we have $C_n \cdot d_{C^0}(\phi,\id) \geq \gamma(\phi) \geq \pi r^2.$ \end{proof}



As first observed in \cite{PolShe}, to a Hamiltonian $H \in \cH_M = C^{\infty}([0,1] \times M,\R),$ normalized by the zero-mean or the compact support condition, one can, via the theory of persistence modules, associate a multi-set of intervals in $\R$ called a barcode. This map is Lipschitz with respect to the $L^{1,\infty}$-distance on $\cH = \cH_M,$ and the bottleneck distance on the space $\barc$ of barcodes. This observation was used in \cite{PolShe}, in \cite{UsherZhang, Zhang, PolSheSto,Team,stevenson,Fraser} and more recently in \cite{KS-bounds, LSV-conj, BHS-spectrum, StoZhang, Usher-BM, RizzSull-pers} to produce various quantitative results in symplectic topology. Set $\barc'$ for the quotient space of $\barc$ with respect to the isometric $\R$-action by shifts.

Denoting for $H \in \cH$ by $\cB'(H) \in \barc'$ its barcode of index $0,$ with coefficients in the Novikov field $\Lambda_{\tmon},$ with quantum variable of degree $(-1),$ considered up to shifts. By \cite[Corollary 6]{KS-bounds}, and Theorem \ref{thm:gamma is C^0 continuous on CP^n}, the barcode $\cB'(H) = \cB'(\phi^1_H)$ depends only on the time-one map $\phi^1_H$ of $H,$ and we immediately obtain the following statement. 

\begin{cor}\label{cor:barcodes extend}
The map $(\Ham(\CP),d_{C^0}) \to (\barc',d'_{\mrm{bottle}}),$ $\phi \mapsto \cB'(\phi)$ is continuous, and hence extends to completions:
\[\cB': \overline{\Ham}(\CP) \to \overline{\barc}'.\]
\end{cor}


In fact, one may take coefficients in $\Lambda_{M,\tmon}$ with quantum variable of degree $(-2),$ in which case the same statement holds for $\cB'(\phi)$ being the image in $\barc'$ of either the index $0$ or the index $1$ barcode of $H \in \cH$. 

In the case of surfaces, a similar statement was proven in \cite{LSV-conj} using different tools, while the same statement was proven in \cite[Remark 8]{KS-bounds} using \cite[Corollary 6]{KS-bounds}. It was also shown in \cite{BHS-spectrum} for closed symplectically aspherical manifolds, using \cite[Corollary 6]{KS-bounds}.

Following \cite{LSV-conj}, we use Corollary \ref{cor:barcodes extend}, and the conjugation invariance property \eqref{eq:conj invariance barcodes} of $\cB',$ to establish that $\cB':\overline{\Ham}(\CP) \to \overline{\barc}'$ is constant on {\em weak conjugacy classes}. Two elements $\phi_0,\phi_1$ of a topological group $G$ are called {\em weakly conjugate} if $\theta(\phi_0) = \theta(\phi_1)$ for all continuous conjugacy-invariant maps $\theta: G \to Y,$ to a Hausdorff topological space $Y.$ This is an equivalence relation, that was studied in ergodic theory and dynamical systems, see \cite{GlasnerWeiss1,GlasnerWeiss2,MannWolff} and references therein. We consider this notion for $G = \overline{\Ham}(\CP).$

It is important to remark the following. For an element $g$ of a topological group $G,$ denote by $\mrm{Conj}(g)$ the conjugacy class of $g$ in $G.$ Then $\phi_0,\phi_1$ are weakly conjugate given that there exist $h_0,\ldots,h_{N+1} \in G$ with $h_0 = \phi_0,$ and $h_N = \phi_1,$ with $\overline{\mrm{Conj}}(h_j) \cap \overline{\mrm{Conj}}(h_{j+1}) \neq \emptyset$ for all $0 \leq j \leq N,$ the closures of the conjugacy classes being taken in $G.$ In particular if $\phi_1$ lies in the closure $\overline{\mrm{Conj}}(\phi_0)$ then $\phi_0,\phi_1$ are weakly conjugate. We refer to \cite{LSV-conj} for further discussion of this notion.
 
%

\begin{cor}\label{cor:weak conj}
The barcode $\cB'(\phi) \in \overline{\barc}'$ for $\phi \in \overline{\Ham}(\CP)$ is a weak conjugacy invariant.
\end{cor}

Now, a small variation on \cite[Proposition 55, Remark 62]{LSV-conj} yields the following. Denote by $r(x,\phi) = \dim_{\bK} HF_{\mrm{loc}}(\phi,x)$ the dimension of the local Floer homology of $\phi$ at a contractible fixed point $x \in \mrm{Fix}_c(\phi)$ (see \cite{Ginzburg-CC}).

\begin{thm}[\cite{LSV-conj}]\label{thm:LSV}
Let the fixed points $\mrm{Fix}_c(\phi)$ of $\phi \in \Ham(\CP)$ in the contractible class be a finite set. Then the barcode $\cB'(\phi)$ consists of a finite number of bars, and the number of endpoints of these bars equals \[\sum_{x \in \mrm{Fix}_c(\phi)} r(x,\phi).\]
\end{thm}

This result together with Corollary \ref{cor:weak conj} implies that $\overline{\Ham}(\CP)$ does not possess a dense conjugacy class, that is - it is not a Rokhlin group. The same consequence for surfaces of higher genus was known by \cite{GambGhys-comm,Ghys-ICM}. The case of the two-torus, as well as that of $(\D^{2n},\om_{std}),$ was settled in \cite{EntPolPy} (the former case building on \cite{Py-torus}, see also \cite{BKS-torus}), and that of the sphere in \cite{SeyfaddiniC0Limits}, while \cite{BHS-spectrum} shows it for closed symplectically aspherical manifolds.

Finally, observing that if $\phi \in \Ham(M,\om)$ is smooth, then the set of endpoints of each representative of $\cB'(\phi)$ in $\barc$ is bounded, implies by Corollary \ref{cor:barcodes extend}, via the example from \cite[Section 6]{LSV-conj}, the following statement.


\begin{cor}\label{cor:non-smooth}
There exists a homeomorphism $\phi \in \overline{\Ham}(\CP)$ that is not weakly conjugate to any diffeomorphism in $\Ham(\CP).$
\end{cor}

\begin{rmk}
	It was recently proven in \cite[Corollary 5.2]{GG-pseudorotations} that for each pseudo-rotation $\phi \in \Ham(\CP),$ that is a Hamiltonian diffeomorphism with precisely $n+1$ periodic points of all positive integer periods, there exists an increasing integer sequence $k_j$ such that $\gamma(\phi^{k_j}) \xrightarrow{j \to \infty} 0.$ Moreover, under a certain strong irrationality assumption on the vector $\Delta \in (\R/\Z)^{n+1}$ given by the mean indices of the periodic points, it is shown in \cite[Theorem 1.4]{GG-pseudorotations} that there is a sequence $k_j$ such that $\phi^{k_j} \xrightarrow{j \to \infty} \id$ in $C^0$ topology. By Theorem \ref{thm:gamma is C^0 continuous on CP^n}, the latter $C^0$ result, whenever it holds, implies the former result on $\gamma.$
\end{rmk}

\begin{rmk}
The argument used for the proofs of Theorems \ref{thm:Viterbo-sharper}, \ref{thm:gamma is C^0 continuous on CP^n}, could yield further $C^0$-continuity results, both absolute and relative (see Theorem \ref{thm:C^0 rel} below) from each new case of Conjecture \ref{conj:Viterbo-generalized}.
\end{rmk}

\subsubsection{Quasimorphisms on the Hamiltonian group of cotangent disk bundles}
The second application of Theorem \ref{thm:Viterbo} is to the symplectic topology of unit cotangent disk bundles $D^*_g L,$ $L \in \cl V.$ These applications were anticipated in \cite{MonznerVicheryZapolsky}. We start with the notion of a quasi-morphism on a group and refer to \cite{scl} for further exposition. 

\begin{df}
	A quasimorphism $\sigma:G \to \R$ on a group $G$ is a function satisfying the bound \[D_{\sigma} = \sup_{x,y \in G}|\sigma(xy) - \sigma(x) - \sigma(y)| < \infty.\]
\end{df}

The number $D_{\sigma}$ is called the {\em defect} of the quasimorphism. If $D_{\sigma} = 0,$ then the quasimorphism is called {\em trivial}: in this case it is in fact a homomorphism $G \to \R.$ For each quasimorphism $\sigma$ there exists a unique {\em homogeneous}, that is additive on each abelian subgroup of $G,$ quasimorphism $\overline{\sigma},$ such that $\sigma - \overline{\sigma}$ is a bounded function. This homogeneization is given by the formula \[\overline{\sigma}(x) = \lim_{k \to \infty} \frac{\sigma(x^k)}{k}.\]  

Quasi-morphisms on the (universal cover of) the Hamiltonian group of closed symplectic manifolds were constructed in \cite{EntovPolterovichCalabiQM} and many subsequent works (we refer to \cite{Entov-ICM} for a review of the literature). However, not many examples are known in the case of open symplectic manifolds \cite{Lanzat-qm,BEP-ball}. One can construct such quasimorphisms by pulling them back by a conformally symplectic embedding $U \to M$ of an open symplectic manifold $U$ into a closed symplectic manifold $M.$ In our case, endowing $L \in \cl V$ with the standard Zoll metric, we observe that each $\epsilon \in (0,1]$ gives a symplectic embedding $\iota_{\epsilon}:D^*_{\epsilon} L \to M\setminus \Sigma,$ of  $D^*_{\epsilon} L = \epsilon \cdot D^*L$ into the respective $M \in \cl W,$ in the complement of $\Sigma.$ Now, as shown in \cite{KS-bounds}, there exists a non-trivial homogeneous Calabi quasimorphism $\sigma:\til{\Ham}(M) \to \R$ on the universal cover of $\Ham(M,\om),$ enjoying the following additional property. If $H \in \cH_M$ is a Hamiltonian with zero mean, with $H_t|_L \equiv c(t)$ for $t \in [0,1],$ then $\sigma([H]) =\int_{0}^{1} c(t) \, dt,$ for the class $[H] \in \til{\Ham}(M,\om)$ generated by the Hamiltonian path $\{\phi^t_H\}_{t \in [0,1]}.$ Now, via $\iota_{\epsilon}$ we obtain a natural homomorphism $i_{\epsilon}:\til{\Ham}_c(D^*_{\eps} L) \to \til{\Ham}_c(M),$ giving a homogeneous quasi-morphism $i_{\eps}^*\sigma = \sigma \circ i_{\eps}$ on $\til{\Ham}_c(D^*_{\eps} L).$  Looking at Hamiltonians $H \in C^{\infty}_c(D^*L,\R)$ with zero mean, and constant on $L$ it is easy to check that $\mu$ is non-trivial. However, it was hitherto unknown whether quasimorphisms $\til{\Ham}_c(D^*L) \to \R$ can be constructed intrinsically from the symplectic geometry of $D^*L$ itself. We resolve this question below for $L \in \cl V.$ Consider the invariant \[\mu: \til{\Ham}_c(D^* L) \to \R\]  \[\mu([H]) = \lim_{k \to \infty} \frac{1}{k} c([L], [H]^{k}).\] Define $\zeta: C^{\infty}_c (D^*L, \R) \to \R$ by $\zeta(H) = \mu([H]),$ the spectral invariants being computed inside $D^*L.$ These maps were defined, and shown to enjoy various properties in \cite[Theorems  1.3 and 1.8, Propositions 1.4 and 1.9]{MonznerVicheryZapolsky}. In particular, $\mu([H])$ depends only on $\phi^1_H,$ and defines a map $\mu: {\Ham}_c(D^*_g L) \to \R.$ We prove, via Theorem \ref{thm:Viterbo}, the following new properties of these maps.


\begin{cor}\label{cor:qm intr}
The map $\mu: \Ham_c(D^* L) \to \R$ is a non-zero homogeneous quasimorphism. Moreover $\mu$ vanishes on each element $\phi \in \Ham_c(D^* L)$ such that $\supp(\phi)$ is displaceable. For $F,G \in C^{\infty}_c (D^*L, \R),$ the map $\zeta$ satisfies \begin{equation}\label{eq:zeta pb}|\zeta(F+G) - \zeta(F) - \zeta(G)| \leq \sqrt{2 C(g,L) ||\{F,G\}||_{C^0}},\end{equation} where $\{F,G\}$ is the Poisson bracket of $F,G.$ In particular, whenever $\{F,G\} = 0,$ we obtain \[\zeta(F+G) = \zeta(F) + \zeta(G).\]
\end{cor}

\section*{Acknowledgements}
I thank Peter Albers, Paul Biran, Octav Cornea, Asaf Kislev, Leonid Polterovich, Vuka\v{s}in Stojisavljevi\'{c}, Dmitry Tonkonog, Renato Vianna, and Frol Zapolsky for fruitful collaborations during which I learnt many of the tools that I apply in this paper. I thank Sobhan Seyfaddini and Georgios Dimitroglou Rizell for useful conversations. This work was initiated and was partially carried out during my stay at the Institute for Advanced Study, where I was supported by NSF grant No. DMS-1128155. It was partially written during visits to Tel Aviv University, and to Ruhr-Universit\"{a}t Bochum. I thank these institutions and Helmut Hofer, Leonid Polterovich, and Alberto Abbondandolo, for their warm hospitality. At the University of Montr\'{e}al, I am supported by an NSERC Discovery Grant and by the Fonds de recherche du Qu\'{e}bec - Nature et technologies.


%% file: prelimVZ.tex

\section{Preliminary notions}
%
%
%

We briefly describe the pertinent part of the standard package of filtered Floer homology in the context of monotone symplectic manifolds, and their monotone Lagragnian submanifolds. We refer to \cite{OhBook,LeclercqZapolsky,KS-bounds} for more details and a review of the literature. However, we emphasize two points. Firstly, in Section \ref{subsec:relations} we describe how filtered relative Hamiltonian Floer homology of a Hamiltonian $H \in \sm{[0,1] \times M, \R}$ and a Lagrangian $L \subset M$ is isomorphic to the filtered Lagrangian Floer homology of the pair of Lagrangian submanifolds $L$ and $L' = (\phi^1_H)^{-1} L,$ with appropriate choices of additional data called anchors \cite[Chapter 14]{OhBook}. Second, in Section \ref{subsec: Floer persistence} we recall and describe a few ways to determine the collection of bar-lengths of the barcodes of persistence modules associated to filtered Floer homology.

\subsection{Filtered Floer homology.}\label{sec:filtered Floer}

All Lagrangian submanifolds $L \subset M$ we consider in this paper shall be weakly {\em homologically monotone} that is the class $\om_L \in H^2(M,L;\R)$ of the symplectic form in cohomology relative to $L$ and the Maslov class $\mu_L \in H^2(M,L;\R)$ are positively proportional \[\om_L = \kappa\cdot \mu_L\] for $\kappa = \kappa_L > 0.$ Moreover, when $M$ is closed, we require that $\ima(\mu_L) = N_L \cdot \Z \subset \Z$ for an integer $N_L \geq 2,$ called the minimal Maslov number of $L$ in $M.$ In this case $M$ will be weakly homologically monotone, that is $[\om] = 2\kappa \cdot c_1(M,\om)$ in $H^2(M;\R).$ When $M$ is not closed, we require it be exact, that is $\om = d\lambda$ for a one-form $\lambda,$ and to have symplectically convex boundary. This means that the vector field $V$ on $M$ defined by $\iota_V \omega = \lambda$ is transverse to $\partial M$ and points outwards at $\partial M.$ In this open case, we shall consider {\em exact} Lagrangian submanifolds, that is $\lambda|_L = df_L,$ for $f_L \in \sm{L,\R}.$ We denote by $\cH \subset C^{\infty}([0,1] \times M,\R)$ the space of time-dependent Hamiltonians on $M,$ where in the closed case $H_t(-) = H(t,-)$ is normalized to have zero mean with respect to $\om^n,$ and in the non-closed case, it is normalized to vanish near the boundary. The time-one maps of isotopies $\{\phi^t_H\}_{t \in [0,1]}$ generated by time-dependent vector fields $X^t_H, \iota_{X^t_H} \omega = - d(H_t),$ are called Hamiltonian diffeomorphisms and form the group $\Ham(M,\om).$ For $H \in \cH$ we call $\overline{H},\til{H} \in \cH$ the Hamiltonians $\overline{H}(t,x) = - H(t,\phi^1_H x),$ $\til{H}(t,x) = - H(1-t,x).$ For $t \in [0,1]$ we have $\phi^t_{\overline{H}} = (\phi^t_H)^{-1},$ while the isotopy $\{\phi^t_{\til{H}}\},$ viewed as a path in $\Ham(M,\om),$ is homotopic to $\{\phi^t_{\overline{H}}\}$ with fixed endpoints. Since homotopic Hamiltonian isotopies give naturally isomorphic graded filtered Floer complexes, we shall identify the two operations $H \mapsto \overline{H},$ and $H \mapsto \til{H}.$ In particular we will identify between $H$ and the two Hamiltonians $\displaystyle\til{\overline{H}} \in \cH,$ $\displaystyle\overline{\til{H}} \in \cH.$ Similarly, for $F,G \in \cH,$ we set $F \# G \in \cH$ to generate the flow $\{\phi^t_F \phi^t_G\}_{t \in [0,1]},$ in other words $F \# G (t,x) = F(t,x) + G(t,(\phi^t_F)^{-1} x).$ A homotopic path is generated by $F \til{\#} G (t,x) = \lambda'_1(t) G(\lambda_1(t),x) + \lambda'_2(t) F(\lambda_2(t),x)$ for surjective monotone non-decreasing reparametrizations $\lambda_1,\lambda_2:[0,1] \to [0,1],$ such that $\supp \lambda'_1 < \supp \lambda'_2.$ Finally, let $\cJ(M,\om)$ be the space of $\om$-compatible almost complex structures on $M.$

In each case below, Floer theory, first introduced by A. Floer \cite{Floer1,Floer2,Floer3}, is a way to set up Morse-Novikov homology for an action functional defined on a suitable cover of a path or a loop space determined by the geometric situation at hand. We refer to \cite{OhBook} and references therein for details on the constructions described in this subsection.

\subsubsection{Absolute Hamiltonian case.} \label{subsec:abs-Ham}


Consider $H \in \cH.$ Let $\cL_{pt} M$ be the space of contractible loops in $M.$ Let $c_M: \pi_1(\cL_{pt} M) \cong \pi_2(M) \to 2 N_M \cdot \Z,$ be the surjection given by $c_M(A) = 2 \left< c_1(M,\om), A\right>.$ Let $\til{\cL}^{\min}_{pt} M = \til{\cL_{pt}} \times_{c_M} (2 N_M \cdot \Z)$ be the cover of $\cL_{pt} M$ associated to $c_M.$ The elements of $\til{\cL}^{\min}_{pt} M$ can be considered to be equivalence classes of pairs $(x,\overline{x})$ of $x \in {\cL}_{pt} M$ and its capping $\overline{x}:\D \to M,$ $\overline{x}|_{\del \D} = x.$ The symplectic action functional \[\cA_{H}: \til{\cL}^{\min}_{pt} M \to \R \] is given by \[\cA_{H}(x,\overline{x}) = \int_0^1 H(t,x(t)) - \int_{\overline{x}} \om,\] that is well-defined by monotonicity: $[\om]= \kappa \cdot c_M.$ Assuming that $H$ is non-degenerate, that is the graph $\mrm{graph}(\phi^1_H) = \{(\phi^1_H(x),x)\,|\, x \in M\}$ intersects the diagonal $\Delta_M \subset M \times M$ transversely, the generators over the base field $\bK = \bF_2$ of the Floer complex $CF(H;J)$ are the lifts $ \til{\cO}(H)$ to $\til{\cL}^{\min}_{pt} M$ of $1$-periodic orbits $\cO(H)$ of the Hamiltonian flow $\{\phi^t_H\}_{t \in [0,1]}.$ These are the critical points of $\cA_{H},$ and we denote by $\Spec(H) = \cA(\til{\cO}(H))$ the set of its critical values. Choosing a generic time-dependent $\om$-compatible almost complex structure $\{J_t \in \cJ(M,\om)\}_{t \in [0,1]},$ and writing the asymptotic boundary value problem on maps $u:\R \times S^1 \to M$ defined by the negative formal gradient on $\cL_{pt} M$ of $\cA_{H},$ the count of isolated solutions, modulo $\R$-translation, gives a differential $d_{H;J}$ on the complex $CF(H;J),$ $d^2_{H;J} = 0.$ This complex is graded by the Conley-Zehnder index $CZ(x,\bar{x})$, with the property that the action of the generator $A = 2N_M$ of $2N_M\cdot \Z$ has the effect $CZ(x,\bar{x} \# A) = CZ(x,\bar{x}) - 2N_M.$ Its homology $HF_*(H)$ does not depend on the generic choice of $J.$ Moreover, considering generic families interpolating between different Hamiltonians $H,H',$ and writing the Floer continuation map, where the negative gradient depends on the $\R$-coordinate we obtain that $HF_*(H)$ in fact does not depend on $H$ either. While $CF_*(H,J)$ is finite-dimensional in each degree, it is worthwhile to consider its completion in the direction of decreasing action. In this case it becomes a free graded module of finite rank over the Novikov field $\Lambda_{M,\tmin} = \bK[q^{-1},q]]$ with $q$ being a variable of degree $(-2N_M).$ 

Moreover, for $a \in \R$ the subspace $CF(H,J)^a$ spanned by all generators $(x,\bar{x})$ with $\cA_{H}(x,\bar{x}) < a$ forms a subcomplex with respect to $d_{H;J},$ and its homology $HF(H)^a$ does not depend on $J.$ Arguing up to $\epsilon,$ one can show that a suitable continuation map sends $FH(H)^a$ to $FH(H')^{a + \cE_{+}(H-H')},$ for \[\cE_{+}(F) = \int_{0}^{1} \max_M(F_t)\,dt.\]
It shall also be useful to define $\cE_{-}(F) = \cE_{+}(-F),\; \cE(F) = \cE_{+}(F) + \cE_{-}(F).$ Finally, one can show that for each $a \in \R,$ $HF(H)^a$ depends only on the class $[H]$ of the path $\{\phi^t_H\}_{t \in [0,1]}$ in the universal cover $\til{\Ham}(M,\om)$ of the Hamiltonian group of $M.$

We mention that it is sometimes beneficial to consider the slightly larger covers $\til{\cL}^{\mrm{mon}}_{pt} = \til{\cL_{pt}} \times_{c_M} (2 \cdot \Z),$ $\til{\cL}^{\mrm{max}}_{pt} = \til{\cL_{pt}} \times_{c_M} \Z,$ defined via the evident inclusions $2N_M \cdot \Z \subset 2 \cdot \Z \subset \Z.$ This corresponds to extending coefficients to $\Lambda_{M,\tmon} = \bK[s^{-1},s]],$ with $\deg(s)=-2,$ and $\Lambda_{\tmon} = \Lambda_{\Delta_M,\tmon} = \bK[t^{-1},t]],$ $\deg(t) = -1,$ respectively.


In case when $H$ is degenerate, we consider a perturbation $\cD = (K^H,J^H),$ with $K^H \in \cH,$ such that $H^{\cD} = H \# K^{H}$ is non-degenerate, and $J^H$ is generic with respect to $H^{\cD},$ and define the complex $CF(H;\cD) = CF(H^{\cD};J^H)$ generated by $\til{\cO}(H;\cD) = \til{\cO}(H^{\cD}),$ and filtered by the action functional $\cA_{H;\cD} = \cA_{H^{\cD}}.$ 

\subsubsection{Relative Hamiltonian case.}\label{subsec:rel-Ham}

Consider $H \in \cH,$ and $L \subset M$ a monotone Lagrangian as above. Let $\cP_{pt}L$ be the space of path from $L$ to $L$ in $M,$ contractible relative to $L.$ Let $\mu_L: \pi_1(\cP_{pt}L) \cong \pi_2(M,L) \to N_L \cdot \Z,$ be the composition of the Hurewicz map with the Maslov class. Let $\til{\cP}^{\min}_{pt} L = \til{\cP_{pt}} \times_{\mu_L} (N_L \cdot \Z)$ be the cover of $\cP_{pt} L$ associated to $\mu_L.$ The elements of $\til{\cP}^{\min}_{pt} L$ can be considered to be equivalence classes of pairs $(x,\overline{x})$ of $x \in {\cP}_{pt} L$ and its capping $\overline{x}:\D \to M,$ $\overline{x}|_{\del \D \cap \{\Im(z) \geq 0\}} = x,$ $\overline{x}(\del \D \cap \{\Im(z) \leq 0\}) \subset L.$ The symplectic action functional \[\cA_{\LH}: \til{\cP}^{\min}_{pt} L \to \R \] is given by \[\cA_{\LH}(x,\overline{x}) = \int_0^1 H(t,x(t)) - \int_{\overline{x}} \om,\] that is well-defined by monotonicity: $\om_L = \kappa \cdot \mu.$ 
Let $\{\phi^t_H\}_{t \in [0,1]}$ be the Hamiltonian flow of $H.$  Assuming $(\LH)$ is non-degenerate, that is $\phi^1_H(L)$ intersects $L$ transversely, the generators over $\bK = \bF_2$ of the Floer complex $CF(\LH;J)$ are the lifts $ \til{\cO}(\LH)$ to $\til{\cP}^{\min}_{pt} L$ consisting of integral trajectories $\cO(\LH)$ of $\{X^t_H\}_{t \in [0,1]}$ with endpoints in $L.$ These are the critical points of $\cA_{\LH},$ and we denote by $\Spec(\LH) = \cA(\til{\cO}(\LH))$ the set of its critical values. Choosing a generic $\{J_t \in \cJ(M,\om)\}_{t \in [0,1]},$ and writing the asymptotic boundary value problem on maps $u:(\R \times [0,1],\R \times \{0\} \cup \R \times \{1\}) \to (M,L)$ defined by the negative formal gradient on $\cP_{pt} L$ of $\cA_{\LH},$ the count of isolated solutions, modulo $\R$-translation, gives a differential $d_{\LH;J}$ on the complex $CF(\LH;J),$ $d^2_{\LH;J} = 0.$ This complex is graded by the Conley-Zehnder (or Robbin-Salamon) index $CZ(x,\bar{x})$, with the property that the action of the generator $A = N_L$ of $N_L\cdot \Z$ has the effect $CZ(x,\bar{x} \# A) = CZ(x,\bar{x}) - N_L.$ Its homology $HF_*(\LH)$ does not depend on the generic choice of $J.$ Moreover, considering generic families interpolating between different Hamiltonians $H,H',$ and writing the Floer continuation map, where the negative gradient depends on the $\R$-coordinate we obtain that $HF_*(\LH)$ in fact does not depend on $H$ either. While $CF_*(\LH;J)$ is finite-dimensional in each degree, it is worthwhile to consider its completion in the direction of decreasing action. In this case it becomes a free graded module of finite rank over the Novikov field $\Lambda_{M,\tmin} = \bK[p^{-1},p]]$ with $p$ being a variable of degree $(-N_L).$ 

Moreover, for $a \in \R$ the subspace $CF(\LH;J)^a$ spanned by all generators $(x,\bar{x})$ with $\cA_{\LH}(x,\bar{x}) < a$ forms a subcomplex with respect to $d_{\LH;J},$ and its homology $HF(\LH)^a$ does not depend on $J.$ Arguing up to $\epsilon,$ one can show that a suitable continuation map sends $HF(\LH)^a$ to $HF(H',L)^{a + \cE_{+}(H-H')},$ for $\cE_{+}(F) = \int_{0}^{1} \max_M(F_t)\,dt.$ Finally, one can show that $HF(\LH)$ depends only on the class $[H] \in \til{\Ham}(M,\om).$ 

We mention that it is sometimes beneficial to consider the slightly larger cover $\til{\cP}^{\mrm{mon}}_{pt} L = \til{\cP}_{pt} L \times_{\mu_L} \Z,$ defined via $N_L \cdot \Z \subset \Z.$ This corresponds to extending coefficients to $\Lambda_{\tmon} = \Lambda_{L,\tmon} = \bK[t^{-1},t]],$ with $\deg(t) = -1.$  

In case when the intersection $\phi^1_H(L) \cap L$ is not transverse, we consider a perturbation $\cD = (K^H,J^H),$ with $K^H \in \cH,$ such that $(H^{\cD},L)$ for $H^{\cD} = H \# K^{H}$ is non-degenerate, and $J^H$ is generic with respect to $(H^{\cD},L),$ and define the complex $CF(\LH;\cD) = CF(H^{\cD},L;J^H)$ generated by $\til{\cO}(\LH;\cD) = \til{\cO}(H^{\cD},L),$ and filtered by the action functional $\cA_{\LH;\cD} = \cA_{H^{\cD},L}.$ 

\subsubsection{Anchored Lagrangian Floer homology.}\label{subsec:anchors}

Following \cite{OhBook} we introduce the following decoration of a Lagrangian submanifold. Fix a base-point $w \in M,$ and a Lagrangian subspace $\lambda_w \subset T_w M.$

\begin{df}\label{def:anchor}
An anchored Lagrangian brane $\underline{L} = (L,\alpha,\lambda)$ is a triple consisting of $L \subset M$ a monotone Lagrangian submanifold, $\alpha,$ which is the data of a point $x \in L,$ and a path $\alpha:[0,1] \to M$ from $\alpha(0) = w$ to $\alpha(1) = x,$ and a section $\lambda$ of the Lagrangian Grassmannian $\alpha^* {\mathcal{L}ag}(M,\om) \to [0,1]$ over $\alpha,$ with $\lambda(0) = \lambda_w.$   
\end{df}
We say that an anchored brane $\ul L = (L,\alpha,\lambda)$ enhances, or decorates $L.$ For two anchored Lagrangians $\underline{L},\underline{L}',$ the path $\alpha_{\ul L, \ul L'} = \overline{\alpha} \# \alpha'$ prescribes a connected component $\cP(\underline{L},\underline{L}') \subset \cP(L,L')$ in the space of paths from $L$ to $L'$, wherein it gives a base-point. One also obtains a natural section $\lambda(\ul L, \ul L')$ of $\alpha_{\ul L, \ul L'}^* {\mathcal{L}ag}(M,\om) \to [0,1].$

Now assume that $\ul L, \ul L'$ are compatible. This means $\kappa = \kappa_L = \kappa_{L'},$ $N = N_L = N_{L'},$ and moreover, the $\bK$-counts $d_{L},d_{L'}$ of Maslov index $2$ $J$-holomorphic disks on $L,$ and $L',$ with respect to a generically chosen $J \in \cJ(M,\om),$ agree: $d_L = d_{L'}.$ Moreover, the Maslov index $\mu_{\ul L, \ul L'}: \pi_1(\cP(\ul L, \ul L')) \to \Z$ takes values in $N \cdot \Z,$ and satisfies $\om_{\ul L, \ul L'} = \kappa \cdot \mu_{\ul L,\ul L'}.$ This will hold in the main situation of interest for us: when $L' = (\phi^1_H)^{-1} L,$ for $H \in \cH.$  

Let $\til{\cP}^{\min}(\ul L, \ul L') = \til{\cP}(\ul L, \ul L') \times_{\mu_{\ul L,\ul L'}} (N \cdot \Z)$ be the cover of $\cP(\ul L, \ul L')$ associated to $\mu_{\ul L,\ul L'}.$ The elements of $\til{\cP}^{\min}(\ul L, \ul L')$ can be considered to be equivalence classes of pairs $(x,\overline{x})$ of $x \in \cP(\ul L, \ul L')$ and its capping $\overline{x}:[0,1] \times [0,1] \to M,$ $\overline{x}({[0,1] \times \{0\}}) \subset L,$  $\overline{x}({[0,1] \times \{1\}}) \subset L',$ $\overline{x}(0,t) = \alpha_{\ul L, \ul L'}(t),$ and
$\overline{x}(1,t) = x(t),$ for $t \in [0,1].$ The symplectic action functional \[\cA_{\ul L, \ul L'}: \til{\cP}^{\min}(\ul L, \ul L') \to \R \] is given by \[\cA_{\ul L, \ul L'}(x,\overline{x}) =  - \int_{\overline{x}} \om,\] that is well-defined by monotonicity: $\om_{\ul L,\ul L'} = \kappa \cdot \mu_{\ul L,\ul L'}.$ 

Assuming that $(L,L')$ is non-degenerate, that is $L$ and $L'$ intersect transversely, the generators over $\bK = \bF_2$ of the Floer complex $CF(\ul L, \ul L';J)$ are the lifts $ \til{L\cap L'}$ of $L\cap L'$ to $\til{\cP}^{\min}(\ul L, \ul L'),$ which are the critical points $\cA_{\ul L, \ul L'}.$ We denote by $\Spec(\ul L,\ul L') = \cA_{\ul L, \ul L'}(\til{L\cap L'})$ the set of its critical values. In this case, choosing a generic $\{J_t \in \cJ(M,\om)\}_{t \in [0,1]},$ and writing the asymptotic boundary value problem on maps $u:(\R \times [0,1],\R \times \{0\},\R \times \{1\}) \to (M,L,L')$ defined by the negative formal gradient on $\til{\cP}^{\min}(\ul L, \ul L')$ of $\cA_{\ul L, \ul L'},$ the count of isolated solutions, modulo $\R$-translation, gives a differential $d_{\ul L , \ul L';J}$ on the complex $CF(\ul L , \ul L';J),$ $d^2_{\ul L , \ul L';J} = 0.$ This complex is graded, via the section $\lambda_{\ul L , \ul L'},$ by the Conley-Zehnder (or Robbin-Salamon) index $CZ(x,\bar{x})$, with the property that the action of the generator $A = N$ of $N\cdot \Z$ has the effect $CZ(x,\bar{x} \# A) = CZ(x,\bar{x}) - N.$ The homology $HF_*(\ul L, \ul L')$ of this complex does not depend on the generic choice of $J.$ While $CF(\ul L , \ul L';J)$ is finite-dimensional in each degree, it is worthwhile to consider its completion in the direction of decreasing action. In this case it becomes a free graded module of finite rank over the Novikov field $\Lambda_{L,L',\tmin} = \bK[p^{-1},p]]$ with $p$ being a variable of degree $(-N).$ 

Moreover, for $a \in \R$ the subspace $CF(\ul L, \ul L';J)^a$ spanned by all generators $(x,\bar{x})$ with $\cA_{\ul L, \ul L'}(x,\bar{x}) < a$ forms a subcomplex with respect to $\cA_{\ul L, \ul L'},$ and its homology $HF(\ul L, \ul L')^a$ does not depend on $J.$ 

We mention that it is sometimes beneficial to consider the slightly larger cover $\til{\cP}^{\mrm{mon}}(\ul L, \ul L') = \til{\cP}(\ul L, \ul L') \times_{\mu_{\ul L, \ul L'}}  \Z,$ defined via $N \cdot \Z \subset \Z.$ This corresponds to extending coefficients to $\Lambda_{\tmon} = \bK[t^{-1},t]],$ with $\deg(t) = -1.$ 

In case when the intersection $L \cap L'$ is not transverse, following \cite{SeidelBook}, we consider a perturbation datum $\cD = (K^{\ul L,\ul L'},J^{\ul L, \ul L'}),$ with $K^{\ul L,\ul L'} \in \cH,$ such that $(\ul L, (K^{\ul L,\ul L'})^*\ul L')$ is non-degenerate (see Secion \ref{subsec:relations} for the pull-back notation), and $J^{\ul L, \ul L'}$ is generic with respect to $(\ul L, (K^{\ul L,\ul L'})^*\ul L'),$  and define the complex $CF(\ul L,\ul L';\cD) = CF(\ul L,\ul L';K^{\ul L,\ul L'},J^{\ul L, \ul L'})$ generated by the lifts $\til{\cO}(\ul L, \ul L';\cD) = \til{\cO}(\ul L, \ul L'; K^{\ul L,\ul L'})$ to $\til{\cP}^{\min}(\ul L, \ul L')$ of Hamiltonian chords $\cO(\ul L, \ul L'; K^{\ul L,\ul L'})$ of $K^{\ul L,\ul L'}$ from $L$ to $L'$ in $\cP(\ul L, \ul L')$    and filtered by the action functional $\cA_{\ul L,\ul L';\cD} = \cA_{L,L',K^{\ul L,\ul L'}},$ given by \begin{equation}\label{eq:anchor-Hamiltonian}\cA_{L,L',K^{\ul L,\ul L'}}(x,\overline{x})= \int_{0}^{1} K^{\ul L,\ul L'}(t,x(t)) \,dt - \int_{\overline{x}} \om .\end{equation}

Finally we remark (see \cite{OhBook,SeidelBook,BiranCorneaS-Fukaya}) that for each triple $\ul L, \ul L', \ul L''$, and $\epsilon>0,$ there exits perturbation data $\cD$ and a product map $CF^{a}(\ul L, \ul L';\cD) \otimes CF^{a'}(\ul L', \ul L'';\cD) \to CF^{a+a'+\epsilon}(\ul L, \ul L'';\cD),$ defined by counting disks with $3$ boundary punctures satisfying a suitable non-linear Cauchy Riemann equation, with boundary conditions on $L,L',L''$ and asymptotic to generators of the three complexes at the punctures. In particular this yields a product map \[HF(\ul L, \ul L') \otimes HF(\ul L', \ul L'') \to HF(\ul L, \ul L'').\]


\subsubsection{Non-Archimedean filtrations and extension of coefficients.}\label{subsec:non-Arch}

Let $\Lambda$ be a field. A non-Archimedean valuation on $\Lambda$ is a function $\nu:\Lambda \to \R \cup \{+\infty\},$ such that  \begin{enumerate}
	\item $\nu(x) = +\infty$ if and only if $x = 0,$
	\item $l(xy) = \nu(x) + \nu(y)$ for all $x,y \in \Lambda,$
	\item $l(x+y) \geq \min\{\nu(x),\nu(y)\},$ for all $x,y \in \Lambda.$
\end{enumerate}
We set $\Lambda_0 = \nu^{-1}([0,+\infty)) \subset \Lambda$ to be the subring of elements of non-negative valuation.  

It will sometimes be convenient to work with a larger coefficient ring in the Floer complexes. The universal Novikov field is defined as \[\Lambda_{\tuniv} = \{\sum_j a_j T^{\lambda_j}\,|\, a_j \in \bK, \lambda_j \to +\infty \}. \]This field possesses a non-Archimedean valuation $\nu: \Lambda_{\tuniv} \to \R \cup \{+\infty\}$ given by $\nu(0) = +\infty,$ and \[\nu(\sum a_j T^{\lambda_j}) = \min\{\lambda_j\,|\,a_j \neq 0 \}.\] 
The fields $\Lambda_{M,\tmin} \subset \Lambda_{M,\tmon}$ embed into $\Lambda_{\tuniv}$ via $s \mapsto T^{2\kappa_M},$ and the fields $\Lambda_{L,\tmin} \subset \Lambda_{L,\tmon}$ embed into $\Lambda_{\tuniv}$ via $t \mapsto T^{\kappa_L}.$ This lets us pull back the valuation on $\Lambda_{\tuniv}$ to a valuation on each one of $\Lambda_{M,\tmin}, \Lambda_{M,\tmon},\Lambda_{L,\tmin}, \Lambda_{L,\tmon}.$ 


Now let $\Lambda$ be a field with non-Archimedean valuation $\nu.$ Following \cite{UsherZhang}, given a finite dimensional $\Lambda$-module $C,$ we call a function $l:C \to \R \cup \{-\infty\}$ a non-Archimedean filtration (function), if it satisfies the following properties: \begin{enumerate}
	\item $l(x) = -\infty$ if and only if $x = 0,$
	\item $l(\lambda x) = l(x) - \nu(\lambda)$ for all $\lambda \in \Lambda, x \in C,$
	\item \label{prop:maximu} $l(x+y) \leq \max\{l(x),l(y)\},$ for all $x,y \in C.$
\end{enumerate}

It is easy to see \cite[Proposition 2.1]{EntovPolterovichCalabiQM}, \cite[Proposition 2.3]{UsherZhang} that the maximum property \eqref{prop:maximu} implies that whenever $l(x) \neq l(y),$ one has in fact \begin{equation}\label{eq:max property filtration} l(x+y) = \max\{l(x),l(y)\}.\end{equation} A $\Lambda$-basis $(x_1,\ldots,x_N)$ of $(C,l)$ is called {\em orthogonal} if \[l(\sum \lambda_j x_j) = \max \{l(x_j) -\nu{\lambda_j} \} \] for all $\lambda_j \in \Lambda.$ It is called {\em orthonormal} if in addition $l(x_j) = 0$ for all $j.$ At this point, we note that a linear transformation $T:C \to C$ with matrix $P \in GL(N,\Lambda_0)$ in an orthonormal basis satisfies $T^*l = l.$ In particular it sends each orthogonal, respectively orthonormal, basis to an orthogonal, respectively orthonormal basis. 

Consider each Floer complex from Sections \ref{subsec:abs-Ham}, \ref{subsec:rel-Ham}, \ref{subsec:anchors} as a finite-dimensional $\Lambda$-module $C,$ for suitable Novikov field $\Lambda.$ The function $\cA:C \to \R \cup \{-\infty\}$ given by $\cA(x) = \inf\{a\,|\, x \in C^a\}$ is a non-Archimedean filtration. It can be computed as follows. Consider a standard basis $x_1,\ldots,x_N$ of $C$ over $\Lambda,$ consisting of arbitrarily chosen lifts of the finite set of periodic orbits, Hamiltonian chords, or Lagrangian intersections involved. Then we have \begin{equation}\cA(\sum \lambda_j x_j) = \max\{\cA(x_j) - \nu(\lambda_j)\}\end{equation} for all $\lambda_j \in \Lambda.$ In other words $x_1,\ldots,x_N$ is an orthogonal basis for $(C,\cA).$ Finally, we note that $d^*\cA \leq \cA,$ and in fact for each $x \in C \setminus \{0\}$ the strict inequality $\cA(d(x)) < \cA(x)$ holds.

To extend coefficients in $C,$ we take \[\overline{C} = C \otimes_{\Lambda} \Lambda_{\tuniv}\] and define a non-Archimedean filtration function $\cA:\overline{C} \to \R\cup \{-\infty\}$ on $\overline{C}$ by declaring that $x_1 \otimes 1,\ldots,x_N \otimes 1$ is an orthogonal basis for $(\overline{C},\cA).$ Finally, we note that the basis $(\overline{x}_1,\ldots,\overline{x}_N) = (T^{\cA(x_1)} x_1, \ldots, T^{\cA(x_N)} x_N)$ is an orthonormal basis of $(\overline{C},\cA)$ that is canonical, in the sense that it does not depend on the ambuguity in the choice of $x_1,\ldots,x_N.$


\subsubsection{Relations between the Floer theories}\label{subsec:relations}
It is first useful to discuss the dependence of $HF(\ul L, \ul L')$ and $HF(\ul L, \ul L')^a$ from Section \ref{subsec:anchors} on Hamiltonian deformations of $L$ and $L'.$ Following \eqref{eq:anchor-Hamiltonian}, we first define for anchored Lagrangians $\ul L, \ul L'$ and $H \in \cH$ the complex $CF(\ul L, \ul L',H;\cD),$ where the Hamiltonian term $K^{\ul L, \ul L',H}$ in $\cD$ can be taken to be identically zero, if $\phi^1(L)$ and $L'$ (alternatively $L$ and $(\phi^1_H)^{-1} L'$) intersect transversely. This complex is filtered by $\cA_{\ul L, \ul L',H; \cD}$ defined as in \eqref{eq:anchor-Hamiltonian}, with $H \# K^{\ul L, \ul L',H}$ replacing the term $K^{\ul L, \ul L'}.$ Now, for $H = 0$ we obtain, after an evident identification, the filtered Lagrangian Floer complex $CF(\ul L, \ul L';\cD)$ as in Section \ref{subsec:anchors}, and for $\ul L = \ul L',$ we obtain the filtered Hamiltonian Floer complex relative to $L,$ as in Section \ref{subsec:rel-Ham}. 
For $H \in \cH$ and anchored brane $\ul L = (L,\alpha,\lambda),$ we define the induced brane to be \[ H_* \ul L = \big(\phi^1_H(L), \alpha \# \left\{\phi^t_H \alpha(1)\right\}, \lambda \# \left\{D_{\alpha(1)}\phi^t_H(\lambda(1))\right\}\big).\] By \cite[Chapter 14]{OhBook} we now have the isomorphisms of filtered graded complexes: \begin{align} \label{eq:naturality right}&CF(\ul L,\ul L',H;\cD) \cong CF(\ul L, \overline{H}_*\ul L';\cD),\\ \label{eq:naturality left}&CF(\ul L,\ul L',H;\cD) \cong CF(H_* \ul L, \ul L';\cD),\end{align} for suitably chosed perturbation data $\cD$ in each case, whose Hamiltonian terms can all be chosen to be as $C^1$-small as necessary. These isomorphisms come essentially from the naturality transformations: \begin{align*}& u(s,t) \mapsto v_2(s,t) = \phi^t_{\overline{H}} u(s,t),\\ & u(s,t) \mapsto v_1(s,t) = \phi^{1-t}_{\overline{H}} u(s,t).\end{align*}

It shall be convenient to denote, for each anchored Lagrangian brane $\ul L,$ \[H^* \ul L := \overline{H}_*\ul L.\]

In particular we obtain the isomorphisms of filtered graded complexes \begin{align} \label{eq:iso right}&CF(\LH;\cD) \cong CF(\ul L, H^*\ul L;\cD),\\ \label{eq:iso left}&CF(\LH;\cD) \cong CF(H_* \ul L, \ul L;\cD),\end{align} for each anchored brane $\ul L$ enhancing the monotone Lagrangian $L.$

In a different direction, one can express the absolute Hamiltonian Floer complex as a special case of the relative Hamiltonian Floer complex for a well-chosen Lagrangian submanifold. More precisely, for a closed weakly monotone symplectic manifold $(M,\om),$ let $M \times M^-$ denote the symplectic manifold $(M\times M, \om \oplus -\om).$ It admits a natural diagonal Lagrangian submanifold $L = \Delta_M \subset M \times M^{-}.$ We note that by smooth reparametrization in the $t$-coordinate on $[0,1] \times M,$ $\lambda: [0,1] \to [0,1],$ $\lambda' \geq 0, \lambda(0) = 0, \lambda(1) = 1$ with $\lambda' \equiv 0$ near $\{0\} \cup \{1\},$ we may assume that our time-dependent Hamiltonians, as well as our time-dependent almost complex structures are in fact periodic, extending smoothly to $\R/\Z \times M,$ and are moreover constant in the $t$ variable in a neighborhood of $\pi([1/2,1]) \times M \subset \R/\Z \times M,$ for $\pi:[0,1] \to \R/\Z$ the natural quotient projection. For such a pair $(H,J)$ there is a canonical isomorphism \cite{LeclercqZapolsky} of filtered graded complexes \begin{equation}\label{eq:Floer abs as rel} CF(H;J) \cong CF(\Delta_M, \hat{H}; \hat{J}) \end{equation} 
where $\hat{H} \in \cH_{M \times M^{-}}$ is defined by \[\hat{H}(t,x,y) = \frac{1}{2}H(t/2, x),\] and $\hat{J}_t \in \cJ(M\times M^{-}, \om \oplus -\om)$ is given by \[\hat{J}_t(x,y) = J_t(x) \oplus -J_0(y).\]

\subsection{Quantum homology, PSS isomorphism, and module structures}\label{subsec:QH}


In this section we describe quantum homology and its relative version. It may be helpful to think of them as absolute and relative Hamiltonian Floer homology, when the Hamiltonian is in fact given by a $C^2$-small, time-independent Morse function. Alternatively, one can consider them as the cascade approach \cite{Frauenfelder} to Morse homology for the unperturbed symplectic area functional on the spaces $\til{\cL}^{\min}_{pt}$ and $\til{\cP}^{\min}_{pt} L.$ While we describe only the algebraic structures pertinent to our arguments, there are other algebraic structures on these Floer complexes, such as pair-of-pants products, and structures of Fukaya categories. For further information on these subjects we refer for example to \cite{SeidelBook,OhBook,LeclercqZapolsky}.

\subsubsection{Quantum homology: absolute case}

Set $QH(M) = H_*(M;\Lambda_{M,\tmin}),$ as a $\Lambda_{M,\tmin}$-module. This module has the structure of a graded-commutative unital algebra over $\Lambda_{M,\tmin}$ whose product, deforming the classical intersection product on homology, is defined in terms of $3$-point genus $0$ Gromov-Witten invariants \cite{McDuffSalamonBIG,Liu-assoc,RuanTian-qh1,RuanTian-qh2,Witten-2d}. The unit for this {\em quantum product} is the fundamental class $[M]$ of $M,$ as in the case of the classical homology algebra. The non-Archimedean filtration $\cA:QH(M) \to \R \cup \{-\infty\}$ is given by declaring $E \otimes 1_{\Lambda},$ for a basis $E$ of $H_*(M,\bK)$ to be an orthonormal basis for $(QH(M),\cA).$


\subsubsection{Quantum homology: relative case}

For a generic triple $(f,\rho,J),$ consisting of  a Morse function $f:L \to \R,$  Riemannian metric $\rho$ on $L,$ and $J \in \cJ_M,$ one defines a deformation $d_{f,\rho,J}$ of the Morse differential $d_{f,\rho} \otimes \id$ on $C(f,\rho; \Lambda_{L,\tmon}) = C(f,\rho; \bK) \otimes_{\bK} \Lambda_{L,\tmon}$ by counting isolated, up to the action of suitable reparametrization groups, {\em pearly trajectories} consisting of configurations of negative gradient trajectories of $f,$ connected by $J$-holomorphic disks with boundary on $L:$ the first trajectory is asymptotic at $s \to -\infty$ to a point $x \in \mrm{Crit}(f),$ and the last trajectory asymptotic at $s \to +\infty$ to a point $y \in \mrm{Crit}(f).$ We remark that the energy of a pearly trajectory is defined to be the sum of $\om$-areas of all the $J$-holomorphic disks that appear in it. The homology of the quantum differential $d_{f,\rho,J}$ is called the Lagrangian quantum homology $QH(L)$ of $L$ \cite{BiranCorneaRigidityUniruling,Bi-Co:qrel-long,BiranCorneaLagrangianQuantumHomology}. The $\Lambda_{L,\tmin}$-module $QH(L)$ has the structure of a unital $\Lambda_{L,\tmin}$-algebra, which is, however, not in general graded-commutative. It inherits a natural non-Archimedean filtration from $\Lambda_{L,\tmin}.$


\subsubsection{Floer homology as a module over quantum homology}

In the absolute case, as discussed in detail in \cite{PolSheSto}, an element $\alpha_M \in QH_m(M) \setminus \{0\}$ gives, for $H \in \cH,$ and $r \in \Z, a \in \R$ a map \[(\alpha_M\ast):HF_{r}(H)^a \to HF_{r+m-2n}(H)^{a+\cA(\alpha_M)}.\] It is in fact a morphism \[(\alpha_M\ast):V_{r}(H) \to V_{r+m-2n}(H)[\cA(\alpha_M)]\] of persistence modules, as defined in Section \ref{subsec: Floer persistence}. This morphism is constructed, in a manner very similar to the quantum cap product (see \cite[Example A.4]{PSS} or \cite{SeidelMCG,Schwarz:action-spectrum,Floer3}) by counting negative $\rho$-gradient trajectories $\gamma:(-\infty,0] \to M$ of a Morse function $f$ on $M,$ for a generic pair $(f,\rho),$ asymptotic at $s \to -\infty$ to critical points of $f,$ and having $\gamma(0)$ incident to Floer cylinders $u:\R \times S^1 \to M$ at $u(0,0).$


In the relative case, as discussed in the filtration-free setting, albeit on the chain level in \cite[Section 5.6.2]{Bi-Co:qrel-long} and \cite[Section 3.5]{Char}, an element $\alpha_L \in QH_m(L) \setminus \{0\}$ gives, for $H \in \cH,$ and $r \in \Z,$ a morphism \[(\alpha_L\ast):V_{r}(\LH) \to V_{r+m-2n}(\LH)[\cA(\alpha_L)]\] of persistence modules. This morphism is constructed, in a manner similar to the absolute case, by counting $(f,\rho,J)$-pearly trajectories, starting from critical points of $f$ and ending at a point of incidence to Floer cylinders $u:\R \times [0,1] \to M$ at $u(0,0).$

A similar definition applies in the case of the Lagrangian Floer homology of the pair $(\ul L, \ul L')$ from Section \ref{subsec:anchors}.

\subsubsection{Relations between Floer homologies, and quantum homology}

It shall be important to point out that in the relative case, the module action of $QH(L)$ on the Floer persistence module $V_*(\LH)$ commutes with the isomorphism \[V_*(\LH) \cong V_*(\ul L, {H}^*\ul L)\] from Section \ref{subsec:relations}.

Furthermore, for generic $J \in \cJ(M,\om),$ having set $\hat{J} = J(x) \oplus - J(y) \in \cJ(M \times M^{-}, \om \oplus - \om),$ we obtain an isomorphism \[\Phi: QH(M) \to QH(\Delta_M) \] noting that the coefficient rings $\Lambda_{M,\tmin}$ and $\Lambda_{\Delta_M,\tmin}$ in fact agree, via the map $p \mapsto q,$ both variables having degree $(-2N_M) = (-N_{\Delta_M}).$ Finally, the module actions of $QH(M)$ on $V_*(H)$ and $QH(\Delta_M)$ on $V_*(\Delta_H, \hat{H})$ agree via the isomorphism $\Phi$ and \eqref{eq:Floer abs as rel}.

\subsubsection{Piunikhin-Salamon-Schwarz isomorphisms}

In the absolute case, one obtains a map $PSS: QH(M) \to HF(H)$ by counting (for generic auxiliary data) isolated configurations of negative gradient trajectories $\gamma:(-\infty,0] \to M$ incident at $\gamma(0)$ with with the asymptotic of $\lim_{s \to -\infty} u(s,-),$ as $s \to -\infty$ of a map $u: \R \times S^1 \to M,$ satisfying a Floer equation \[\del_s \,u + J_t(u) \,(\del_t \,u - X^t_{K}(u)) = 0, \] where for $(s,t) \in \R \times S^1,$ $K(s,t) \in \sm{M,\R}$ is a small perturbation of $\beta(s) H_t,$ coinciding with it for $s \ll -1$ and $s \gg +1,$ and $\beta: \R \to [0,1]$ is a smooth function satisfying $\beta(s) \equiv 0$ for $s \ll -1$ and $\beta(s) \equiv 0$ for $s \gg +1.$ This so-called Piunikhin-Salamon-Schwarz map \cite{PSS} is an isomorphism of $\Lambda_{M,\tmin}$-modules, which in fact intertwines the quantum product on $QH(M)$ with the pair of pants product in Hamiltonian Floer homology.

In the relative case, one obtains a map $PSS: QH(L) \to HF(\LH)$ by counting (for generic auxiliary data) isolated configurations of pearly trajectories from critical points of a Morse function $f,$ to a point of incidence with $\lim_{s \to -\infty} u(s,-)$ of a map \[u: (\R \times [0,1],\R \times \{0\} \cup \R \times \{1\}) \to (M,L),\] satisfying a Floer equation $\del_s \,u + J_t(u) \,(\del_t \,u - X^t_{K}(u)) = 0,$ where for $(s,t) \in \R \times [0,1],$ $K(s,t) \in \sm{M,\R}$ is a small perturbation of $\beta(s) H_t$ as above. This so-called Lagrangian Piunikhin-Salamon-Schwarz map (see \cite{Bi-Co:qrel-long,BiranCorneaRigidityUniruling},\cite[Section 3.4]{Char}, \cite{Zap:Orient} and references therein) is an isomorphism of $\Lambda_{L,\tmin}$-modules, and in fact intertwines the quantum product on $QH(L)$ with the Lagrangian pair of pants product in relative Hamiltonian Floer homology.

%

Finally, it is worthwhile to note that the quantum product maps \[QH(M) \otimes QH(M) \to QH(M), \] \[QH(L) \otimes QH(L) \to QH(L), \] are isomorphic to the module action maps \[QH(M) \otimes HF(H) \to HF(H), \] \[QH(L) \otimes HF(\LH) \to HF(\LH),\] via the isomorphisms $\id \otimes PSS$ on the left hand side, and $PSS$ on the right hand side. 


\subsection{Spectral invariants}\label{subsec:spec}



Given a filtered complex $(C,\cA),$ to each homology class $\alpha \in H(C)$, denoting by $H(C)^a = H(C^a),$ $C^a = \cA^{-1} (-\infty,a),$ we define
a spectral invariant by \[c(\alpha, (C, \cA)) = \inf\{a \in \R\,|\, \alpha \in \ima(H(C)^a \to H(C)) \}.\] In the case of $(C,\cA) = (CF(\ul L, \ul L';\cD),\cA_{\ul L, \ul L';\cD})$ we denote $c(\alpha, \ul L, \ul L'; \cD) = c(\alpha,(C,\cA)).$ In the two cases of Hamiltonian Floer homology, one can obtain homology classes by the PSS isomorphism. This lets us define spectral invariants by: 
\[c(\alpha_M, H; \cD) = c(PSS(\alpha_M),(CF(H;\cD),\cA_{H;\cD})),\]
\[c(L;\alpha_L, H; \cD) = c(PSS(\alpha_L),(CF(H;\cD),\cA_{H;\cD})),\]
for $\alpha_M \in QH(M),\,\alpha_L \in QH(L).$ 
From the definition it is clear that the spectral invariants do not depend on the almost complex structure term in $\cD.$ Moreover, if $H,$ $(\LH),$ or $(\ul L, \ul L')$ are non-degenerate, we may choose the Hamiltonian term in $\cD$ to vanish identically, and denote the resulting invariants by:\[c(-,H),\; c(L;-,H),\; c(-,(\ul L, \ul L')).\] Moreover, by \cite[Section 5.4]{BiranCorneaRigidityUniruling} spectral invariants remain the same under extension of coefficients, hence below we do not specify the Novikov field $\Lambda$ that we work over. Spectral invariants enjoy numerous useful properties, the relevant ones of which we summarize here:

\begin{enumerate}
\item {\em spectrality:} for each $\alpha_M \in QH(M) \setminus \{0\},$ $\alpha_L \in QH(L) \setminus \{0\},$ $\alpha_{\ul L, \ul L'} \in HF(\ul L, \ul L')$ and $H \in \cH,$ \[c(\alpha_M, H) \in \Spec(H),\; c(\alpha_L, H) \in \Spec(\LH),\; c(\alpha_{\ul L, \ul L'}) \in \Spec(\ul L, \ul L').\]
\item {\em non-Archimedean property:} $c(-,H;\cD),\; c(L;-,H;\cD),\;c(-,\ul L, \ul L';\cD)$ are non-Archimedean filtration functions on $QH(M),$ $QH(L),$ $HF(\ul L, \ul L'),$ as modules over the Novikov field $\Lambda$ with its natural valuation. 
\item {\em continuity:} for each $\alpha_M \in QH(M) \setminus \{0\},$ $\alpha_L \in QH(L) \setminus \{0\},$ and $F,G \in \cH,$
\[|c(\alpha_M,F) - c(\alpha_M,G)| \leq \cE(F-G),\] \[|c(L;\alpha_L,F) - c(L;\alpha_L,G)| \leq \cE(F-G)\]
\item {\em triangle inquequality:} for each $\alpha_M,\alpha'_M \in QH(M),$ $\alpha_L, \alpha'_L \in QH(L),$ $F,G \in \cH,$ and $\alpha' \in HF(\ul L, \ul L'),$  $\alpha'' \in HF(\ul L', \ul L''),$

\[c(\alpha_M \ast \alpha'_M,F\#G) \leq c(\alpha_M ,F) + c(\alpha'_M ,G),\]
\[c(\alpha_L \ast \alpha'_L,F\#G) \leq c(\alpha_L ,F) + c(\alpha'_L ,G),\]
\[c( \alpha' \ast \alpha'',(\ul L, \ul L'')) \leq c(\alpha' ,(\ul L, \ul L')) + c(\alpha'' ,(\ul L', \ul L'')).\]



\end{enumerate}

In the abstract case of a complex $(C,d)$ over $\Lambda$ filtered by $\cA,$ the non-Archimedean property lets us consider the spectral invariant map, as an {\em induced non-Archimedean filtration function} \[\mrm{H}(\cA) : H(C,d) \to \R \cup \{-\infty\}.\] 

We remark that the key part of the non-Archimedean property, the maximum property, is called the {\em characteristic exponent} property in \cite{EntovPolterovichCalabiQM}. Moreover, we note that by Sections \ref{subsec:relations} and \ref{subsec:QH}, for $H \in \cH$ and $a \in QH(M),$ we have the following identity of spectral invariants: \[c(\Delta_M; \Phi(a), \hat{H}) = c(a,H),\] 
\[c(\Delta_M; \Phi(a), [H] \times \til{\id}) = c(a,[H]).\] 

\subsubsection{Spectral norm.}\label{subsubsec:spec norm}
For $H \in \cH$ we define its spectral pseudo-norm by \[\gamma(H) = c([M],H) + c(M,\overline{H}),\] which depends only on $[H],$ and by a result of \cite{Oh-specnorm} (see also \cite{Usher-sharp,McDuffSalamonBIG}) gives the following non-degenerate spectral norm $\gamma:\Ham(M,\om) \to \R_{\geq 0},$ \[\gamma(\phi) = \inf_{\phi^1_H = \phi} \gamma(H),\] and hence spectral distance $\gamma(\phi,\phi') = \gamma(\phi' \phi^{-1}).$ Similarly for $H \in \cH,$ and monotone $L \subset M$ with $QH(L) \neq 0,$ define \[\gamma(\LH) = c(L;[L],H) + c(L;[L],\overline{H}),\] then by \cite{Viterbo-specGF,KS-bounds} \[\gamma(L',L) = \inf_{\phi^1_H(L) = L'} \gamma(\LH),\] together with invariance under the left action of $\Ham(M,\om),$ defines a non-degenerate distance on the Hamiltonian orbit $\cO_L = \Ham(M,\om) \cdot L$ of $L.$ Finally, for $L' \in \cO_L,$ where $L = 0_{Q} \subset T^* Q$ is the zero-section in the cotangent bundle of a closed manifold $Q,$ the spectral norm $\gamma(L,L')$ can be reformulated as follows. Consider $\ul L$ decorating $L,$ and take $\ul L' = H^* \ul L$ decorating $(\phi^1_H)^{-1} L.$ Then \begin{equation}\label{eq:gamma in cotangent via x,y}\gamma(L,L') = c(x, (\ul L,\ul L')) - c(y, (\ul L, \ul L')),\end{equation} for homogeneous elements $x,y \in HF(\ul L, \ul L') \setminus \{0\}$  such that \begin{equation}\label{eq:action of pt x y}[pt] \ast x = y.\end{equation} 

We note that in view of \cite{Abouzaid-nearbyMaslov,FukayaSeidelSmith-cotangentsc,FukayaSeidelSmith-cotangentcat} the same reformulation allows one to define a distance function $\gamma(L,L')$ for any exact Lagrangian $L' \subset T^*L.$ Indeed, $L'$ is isomorphic in the suitably defined Fukaya category to $L,$ so that this isomorphism commutes with the action of $QH(L) \cong H_*(L).$ This implies in particular that for each brane $\ul L$ decorating $L,$ there exists a brane $\ul L'$ decorating $L',$ with $HF(\ul L, \ul L') \cong H_*(L),$ commuting with the action of $H_*(L).$ In this case $x, y$ are determined by \eqref{eq:action of pt x y} uniquely up to multiplication by $\bK \setminus\{0\},$ and we define $\gamma(L,L')$ by \eqref{eq:gamma in cotangent via x,y}. We note that by Poincar\'{e} duality for spectral invariants (see e.g. \cite{EntovPolterovichCalabiQM,LeclercqZapolsky}), this definition is equivalent to \[\gamma(L,L') = \inf c(x,(\ul L, \ul L')) + c(y, (\ul L',\ul L))\] where the infimum runs over all $\ul L, \ul L'$ decorating $L,L'$ and $x \in HF(\ul L, \ul L'),$ $y \in HF(\ul L', \ul L)$ with $x \ast y = u_{L} \in HF(\ul L, \ul L)$ corresponding to $[L]$ under $HF(\ul L, \ul L) \cong QH(L) \cong H_*(L),$ and $y \ast x = u_{L'} \in HF(\ul L', \ul L')$ corresponding to $[L']$ under $HF(\ul L', \ul L') \cong QH(L') \cong H_*(L').$ For further descriptions of metric structures coming from isomorphisms in Fukaya categories, we refer to \cite{BiranCorneaS-Fukaya}.


\subsection{Floer persistence}\label{subsec: Floer persistence}

\subsubsection{Rudiments of persistence modules}

Let $\mrm{Vect}_{\bK}$ denote the category of finite-dimensional vector spaces over $\bK,$ and $(\R,\leq)$ denote the poset category of $\R.$ A {\em persistence module} over $\bK$ is a functor \[V:(\R,\leq) \to \mrm{Vect}_{\bK}.\] In other words $V$ consists of a collection $\{V^a \in \mrm{Vect}_{\bK}\}_{a \in \R}$ and $\bK$-linear maps $\pi_V^{a,a'}:V^a \to V^{a'}$ for each $a \leq a',$ satisfying $\pi_V^{a,a} = \id_{V^a},$ and $\pi_V^{a',a''} \circ \pi_V^{a,a'} = \pi_V^{a,a''}$ for all $a\leq a'\leq a''.$ These functors with their natural transformations form an abelian category \[Fun((\R,\leq),\mrm{Vect}_{\bK}),\] where $A \in \hom(V,W)$ consists of a collection $\{A^a \in \hom_{\bK}(V^a,W^a)\}_{a \in \R}$ that commutes with the maps $\pi_{V}^{a,a'}, \pi_{V}^{a,a'},$ for each $a \leq a'.$ We require the following further technical assumptions, that hold in all our examples: \begin{enumerate}
	\item {\em support:} $V^a = 0$ for all $a \ll 0.$
	\item {\em finiteness:} there exists a finite subset $S \subset \R,$ such that for all $a,a'$ in each connected component of $\R \setminus S,$ the map $\pi_V^{a,a'}:V^a \to V^{a'}$ is an isomorphism.
	\item {\em continuity:} for each two consecutive elements $s_1< s_2$ of $S,$ and $a \in (s_1,s_2),$ the map $\pi^{a,s_2}: V^a \to V^{s_2}$ is an isomorphism.
\end{enumerate}

Persistence modules with these properties form a full abelian subcategory \[\pemod \subset Fun((\R,\leq),(\mrm{Vect}_{\bK})).\]

The {\em normal form theorem} \cite{CarlZom,CrawBo} for persistence modules states that the isomorphism class of $V \in \pemod$ is classified by a finite multiset $\cB(V) = \{(I_k,m_k)\}_{1 \leq k \leq N'}$ of intervals $I_k \subset \R,$ where $I_k = (a_k,b_k)$ for $k \in (0,K] \cap \Z,$ and $I_k = (a_k,\infty) $ for $k \in (K,N'] \cap \Z$ for some $0 \leq K= K(V) \leq N'.$ We denote $B=B(V) = N'-K \geq 0.$  The intervals are called bars, and a multiset of bars is called a barcode. The bar lengths are defined as $|(a_k,b_k)| = b_k - a_k,$ and $|(a_k,\infty)| = +\infty.$

The {\em isometry theorem} for persistence modules \cite{CdSGO-structure,BauLes,CCSGGO-proximity}, culminating the active development initiated in \cite{CEH-stability},  states the fact that the barcode map \[\cB:(\pemod,d_{\mrm{inter}}) \to (\barc,d_{\mrm{bottle}})\] \[V \mapsto \cB(V)\] is {\em isometric} for the following two distances. 

The {\em interleaving distance} between $V,W \in \pemod$ is given by \begin{align*}d_{\mrm{inter}}(V,W) = \inf \{\delta > 0\;|\; \exists &f \in \hom(V,W[\delta]), g \in \hom(W,V[\delta]), \\ &g[\delta]\circ f = sh_{2\delta,V}, f[\delta]\circ g = sh_{2\delta,W} \},\end{align*} where for $V \in \pemod,$ and $c \in \R,$ $V[c] \in \pemod$ is defined by pre-composition with the functor $T_c: (\R,\leq) \to (\R,\leq), t\to t+c,$ and for $c \geq 0,$ $sh_{c,V} \in \hom (V,V[c])$ is given by the natural transformation $\id_{(\R,\leq)} \to T_c.$ The pair $f,g$ from the definition is called a {\em $\delta$-interleaving}.

The {\em bottleneck distance} between $\cB,\cC \in \barc$ is given by \[d_{\mrm{bottle}}(\cB,\cC) = \inf \bra{\delta > 0\,|\;\exists\; \delta-\text{matching between} \;\cB,\cC },\] where a $\delta$-matching between $\cB,\cC$ is a bijection $\sigma:\cB^{2\delta} \to \cC^{2\delta}$ between two sub-multisets $\cB^{2\delta} \subset \cB,$ $\cC^{2\delta} \subset \cC,$ each containing all the bars of length $> 2\delta$ of $\cB,$ $\cC$ respectively, such that if $\sigma((a,b)) = (a',b')$ then $|a-a'|\leq \delta,$ $|b-b'| \leq \delta.$ 

Finally, we record the the quotient space $(\barc',d'_{\mrm{bottle}})$  of $(\barc,d_{\mrm{bottle}})$ by the isometric $\R$-action by shifts: $c \in \R$ acts by $\cB = \{(I_k,m_k)\} \mapsto \cB[c] = \{(I_k - c, m_k)\},$ and $d'_{\mrm{bottle}}([\cB],[\cC]) = \inf_{c\in \R} d'_{\mrm{bottle}}(\cB,\cC[c])$ for $\cB,\cC \in \barc.$ Note that bar-lengths give a well-defined map from $\barc'$ to multi-subsets of $\R_{>0} \cup \bra{+\infty}.$

\subsubsection{Floer persistence: interleaving, invariance, spectral norms}

As remarked in Section \ref{sec:filtered Floer} for each $r \in \Z,$ the degree $r$ subspace $C_r$ of the Floer complex $C$ considered therein is finite-dimensional over the base field $\bK.$ This implies that the degree $r$ homology $H_r(C)^a = H_r(C^a)$ is in $\mrm{Vect}_{\bK}$ for all $a \in \R.$ Furthermore, inclusions $C^a \to C^{a'}$ of graded complexes for $a \leq a',$ yield maps $\pi^{a,a'}:H_r(C)^a \to H_r(C)^{a'}.$ As it was first observed in \cite{PolShe} (see also \cite{PolSheSto}), the collection $V_r(C)$ of the vector spaces $\{H_r(C)^a\}_{a\in \R},$ and maps $\{\pi^{a,a'}\},$ constitutes an object of the category $\pemod.$ We will denote these persistence modules by $V_r(H;\cD), V_r(\LH;\cD), V_r(\ul L, \ul L';\cD)$ in general, and by $V_r(H), V_r(\LH), V_r(\ul L, \ul L'),$ when $H, (\LH), (\ul L, \ul L')$ are non-degenerate. We use these notations interchangeably, with the understanding that the former is used in the degenerate case, where the Hamiltonian terms in the perturbation data is considered to be as $C^2$-small as necessary, and the latter is used in the non-degenerate case.


Finally, by \cite{PolShe,PolSheSto} Floer continuation maps induce $\cE(H-H')$-interleavings between the pairs $V_r(H),V_r(H')$ and $V_r(\LH),V_r(H',L).$ This implies that \begin{align}\label{eq: barc Hofer Lip} & d_{int}(V_r(H),V_r(H')) \leq \til{d}_{\mrm{Hofer}}([H],[H']) \\ \notag &{d}_{int}(V_r(\LH),V_r(H',L)) \leq \til{d}_{\mrm{Hofer}}([H],[H']),\end{align} where $\til{d}$ is a pseudo-metric on $\til{\Ham}(M,\om)$ defined by \[\til{d}([H],[H']) = \inf \cE(F-G),\] the infimum running over all $F,G \in \cH$ with $[F] = [H],$ $[G] = [H'].$

We recall the ring $\Lambda_{\tmon} = \bK[t^{-1},t]]$ with variable $t$ of degree $(-1).$ For $H \in \cH,$ $L \subset M$ a monotone Lagrangian, and $\ul L, \ul L'$ two anchored Lagrangian branes, consider the associated Floer persistence modules $V_0(H),V_0(\LH),V_0(\ul L, \ul L')$ of degree $0$ with coefficients in $\Lambda_{\tmon}.$ Let $\cB_0(H), \cB(\LH), \cB(\ul L, \ul L')$ be the corresponding barcodes.  By \cite{KS-bounds} (see also \cite[Propositions 5.3, 6.2]{UsherBD2}) the images $\cB'_0(\phi^1_H), \cB'(\phi^1_H(L),L), \cB'(L, L', [\alpha_{\ul L, \ul L'}])$ of these barcodes in $(\barc',d'_{\mrm{bottle}})$ depend only on $\phi^1_H, (\phi^1_H(L),L),$ and $(L,L',[\alpha_{\ul L, \ul L'}])$ respectively, where $[\alpha_{\ul L, \ul L'}] \in \pi_0(\cP(L,L'))$ is the free homotopy class of $\alpha_{\ul L, \ul L'}$ in $\cP(L,L').$ We note that by \eqref{eq:iso right} $\cB'_0(\phi^1_H (L),L) = \cB'_0(L,(\phi^1_H)^{-1} L,[\alpha_{\ul L, H^* \ul L}]) = \cB'_0(\phi^1_H L,L,[\alpha_{H_* \ul L, \ul L}])$  in $\barc',$ for each $\ul L$ decorating $L.$ Finally, by a change of coordinates given by $\psi \in \Symp(M,\om),$ there is an identity of barcodes \begin{align}\label{eq:conj invariance barcodes}\cB'(\phi) = \cB'(\psi \phi \psi^{-1}), \; \cB'(L,L') = \cB'(\psi L, \psi L'),\\ \notag \cB'(L, L', [\alpha_{\ul L, \ul L'}]) = \cB'(\psi L, \psi L', [\psi \alpha_{\ul L, \ul L'}]).\end{align}


We define the bar-length spectrum of $\phi^1_H, (L, \phi^1_H(L)),$ and $(L,L',[\alpha_{\ul L, \ul L'}])$ to coincide with the corresponding sub-multisets of $\R_{>0} \cup \bra{+\infty}$ arranged as increasing sequences, taking into account multiplicities. We define the {\em boundary depth} $\beta(\phi), \beta(L, L')$ and $\beta(L,L',[\alpha_{\ul L, \ul L'}])$ of, respectively, $\phi, (L, L')$ with $L' \in \cO_L,$ and $(L,L',[\alpha_{\ul L, \ul L'}]),$ to be the {\em maximal length of a finite bar} in the associated barcodes. This notion was first introduced by Usher \cite{UsherBD1,UsherBD2} in different terms, and shown to satisfy various properties, including the above invariance statement.

Finally, in view of \eqref{eq: barc Hofer Lip}, we obtain for $\phi, \psi \in \Ham(M,\om),$ and $L',L'' \in \cO_L,$ \begin{align}\label{eq: barc' Hofer Lip} & d'_{\mrm{bottle}}(\cB'(\phi),\cB'(\psi)) \leq {d}_{\mrm{Hofer}}(\phi,\psi) \\ \label{eq: barc' Hofer Lip rel} &{d}'_{\mrm{bottle}}(\cB'(L,L'),\cB'(L,L'')) \leq {d}_{\mrm{Hofer}}(L',L''),\end{align} where \[d_{\mrm{Hofer}}(\phi,\psi) = \inf_{\phi^1_F = \phi, \phi^1_G = \psi} \til{d}([F],[G])\] is the celebrated Hofer metric \cite{HoferMetric,Lalonde-McDuff-Energy} on $\Ham(M,\om),$ and \[d_{\mrm{Hofer}}(L',L'') = \inf_{\phi(L') = L''} d_{\mrm{Hofer}}(\id, \phi) \] is Chekanov's Lagrangian Hofer metric \cite{ChekanovFinsler}. The method of filtered continuation elements introduced in \cite{ShelukhinHZ}, with inspiration from \cite{BiranCorneaS-Fukaya,AK-simplehomotopy}, was used to improve \eqref{eq: barc' Hofer Lip} to \begin{equation}\label{eq: barc' spectral Lip}
d'_{\mrm{bottle}}(\cB'(\phi),\cB'(\psi)) \leq \frac{1}{2}{\gamma}(\phi,\psi). 
\end{equation}
This was extended to the relative setting, reproving \eqref{eq: barc' spectral Lip}, in \cite{KS-bounds}, and showing the following extension of \eqref{eq: barc' Hofer Lip rel}: \begin{equation}\label{eq: barc' spectral Lip rel}{d}'_{\mrm{bottle}}(\cB'(L,L'),\cB'(L,L'')) \leq \frac{1}{2}{\gamma}(L',L''),\end{equation} for $L', L'' \in \cO_L,$ with the assumption that $L$ is {\em wide} \cite{BiranCorneaRigidityUniruling}, that is $QH(L) \cong H_*(L;\Lambda),$ as $\Lambda$-modules.
 
Furthermore, we note that by Section \ref{subsec:relations}, for $H \in \cH_M,$ and all $r\in \Z,$ the persistence modules $V_r(\hat{H},\Delta_M)$ and $V_r(H)$ agree. Hence \[\cB'((\phi \times \id) \Delta_M, \Delta_M) = \cB'(\phi)\] for all $\phi \in \Ham(M,\om).$ In particular, $\beta((\phi \times \id) \Delta_M, \Delta_M) = \beta(\phi).$

 
We conclude this section by observing that \eqref{eq: barc' Hofer Lip} and \eqref{eq: barc' Hofer Lip rel} imply that the boundary depth $\beta(\phi)$ is $1$-Lipschitz in the Hamiltonian spectral norm, while $\beta(L,L')$ is $1$-Lipschitz in the Lagrangian spectral norm in the $L'$ variable. In particular $\beta(\phi),$ similarly to $\gamma(\phi),$ is defined for arbitrary $\phi \in \Ham(M,\om),$ and $\beta(L,L'),$ similarly to $\gamma(L,L'),$ is defined for arbitrary $L' \in \cO_L.$  
 
\subsubsection{Bar-lengths, extended coefficients, and torsion exponents.}


We note the following two alternative descriptions of the bar-length spectrum. Firstly, consider each one of the relevant Floer complexes $(C,d)$ over $\Lambda = \Lambda_{\tmin},$ filtered by $\cA$ as in Section \ref{subsec:non-Arch}. By \cite{UsherZhang}, the complex $(C,d)$ admits an orthogonal basis \[E = (\xi_1,\ldots,\xi_{B},\eta_1,\ldots,\eta_K,\zeta_1,\ldots,\zeta_K)\] such that $d\xi_j = 0$ for all $j \in (0,B] \cap \Z,$ and $d \zeta_j = \eta_j$ for all $j \in (0,K] \cap \Z.$ The finite bar-lengths are then given by $\bra{\beta_j = \beta_j(C,d) = \cA(\zeta_j) - \cA(\eta_j)}$ for $j \in (0,K] \cap \Z,$ which we assume to be arranged in increasing order, while there are $B$ infinite bar-lengths, corresponding to $\xi_j$ for $j \in (0,B] \cap \Z.$ We note that this description yields the identity $N = B + 2K,$ where the numbers $N,B,K$ can be computed via $N = \dim_{\Lambda} C,$ $B = \dim_{\Lambda} H(C,d),$ and $K = \dim \ima(d).$

Extending coefficients to $\Lambda_{\tuniv},$ we can further normalize the basis $E$ to obtain the orthonormal basis \begin{align*}\overline{E} &= (\bar{\xi}_1,\ldots,\bar{\xi}_{B},\bar{\eta}_1,\ldots,\bar{\eta}_K,\bar{\zeta}_1,\ldots,\bar{\zeta}_K) = \\ &= (T^{\cA(\xi_1)}\xi_1,\ldots,T^{\cA(\xi_B)}\xi_{B},T^{\cA(\eta_1)}\eta_1,\ldots,T^{\cA(\eta_K)}\eta_K,T^{\cA(\zeta_1)}\zeta_1,\ldots,T^{\cA(\zeta_K)}\zeta_K).\end{align*} This basis satisfies $d\bar{\xi}_j = 0$ for all $j \in (0,B] \cap \Z,$ and $d \bar{\zeta}_j = T^{\beta_j} \bar{\eta}_j$ for all $j \in (0,K] \cap \Z.$ We note that passing to this basis has the following computational advantage. Fist, each two orthonormal bases are related by a linear transformation $T:\overline{C} \to \overline{C}$ with matrix in $GL(N,\Lambda_{\tuniv,0}).$ Second, to compute the bar-length spectrum, it is sufficient to consider the matrix $[d]$ of $d:\overline{C} \to \overline{C}$ in any orthonormal basis, for example the canonical one from Section \ref{subsec:non-Arch}, and bring it to Smith normal form over $\Lambda_{\tuniv,0}.$ The diagonal coefficients, in order of increasing valuations, will be $\{T^{\beta_j}\}.$ We remark that while $\Lambda_{\tuniv,0}$ is not a principal ideal domain, each of its finitely generated ideals is indeed principal, and therefore Smith normal form applies in this case. 

It shall be important to remark that the considerations regarding the Smith normal form and the bar-length spectrum apply to the case of arbitrary complexes $(C,d)$ over $\Lambda$ with non-Archimedean filtration function $\cA.$ In fact given a filtered map $D:(C,\cA) \to (C',\cA')$ between two filtered $\Lambda$-modules, that is $D^* \cA' \leq \cA,$ it is shown in \cite{UsherZhang} that $D$ has a {\em non-Archimedean spectral value decomposition}: that is orthogonal bases $E = E_{\mrm{coim}} \sqcup E_{\ker}$ of $(C,\cA),$ and $E' = E_{\ima} \sqcup E_{\mrm{coker}}$ of $(C',\cA')$ such that $D(E_{\ker}) = 0,$ while $D|_{E_{\mrm {coim}}}:E_{\mrm{coim}} \xrightarrow{\sim} E_{\mrm{im}}$ is an isomorphism of sets. The spectral values consist of the numbers $\beta_e = \cA(e) - \cA'(D(e))\geq 0$ for $e \in E.$ Over $\Lambda = \Lambda_{\tuniv},$ we may instead ask for $E,E'$ to be orthonormal, and require that for each $e \in E_{\mrm{coim}},$ there exists $e' \in E_{\mrm{im}},$ and $\beta_e > 0,$ such that $D(e) = T^{\beta_e} e',$ and $D(E_{\ker}) = 0.$ It is easy to see that the spectral values correspond directly to the bar-lengths for the filtered complex \[(Cone(D),\cA \oplus \cA'),\] given as a $\Lambda$-module by $C \oplus C',$ with filtration $\cA \oplus \cA' = \max\{\cA,\cA'\},$ and differential \[d_{Cone}(c,c') = (-d_C(c), D(c) + d_{C'}(c') ).\]

Finally, following \cite{FOOO-polydiscs}, it is easy to see that the matrix of the differential $d$ in the canonical basis from Section \ref{subsec:non-Arch} has all coefficients in $\Lambda_{\tuniv,0}.$ Indeed, the coefficient of $\overline{x_i}$ in $d\overline{x}_j$ is given by $\left<d\bar{x}_j, x_i\right> = \sum T^{E(u)} \in \Lambda_{\tuniv,0},$ where $E(u)$ is the energy of $u$ as a negative gradient trajectory of the corresponding action functional, and the sum runs over all isolated (modulo $\R$-translations) negative gradient trajectories asymptotic to $x_j$ at times $s \to -\infty,$ and to $x_i$ at times $s \to +\infty.$ Therefore, one can define the Floer complex, in each of the three cases from Section \ref{sec:filtered Floer}, with coefficients in $\Lambda_{\tuniv,0}.$ Its homology will be a finitely generated $\Lambda_{\tuniv,0}$-module, and will therefore have the form $\mathcal{F} \oplus \cl{T},$ where $\cl F$ is a free $\Lambda_{\tuniv,0}$-module, and $\cl T$ is a torsion $\Lambda_{\tuniv,0}$-module. The bar-lengths in this setting are given by the identity \[\cl T \cong \bigoplus_{1 \leq j \leq K}\Lambda_{\tuniv,0}/(T^{\beta_j}).\]

We refer to \cite{UsherZhang,KS-bounds} for more details of the identifications between the various descriptions of the bar-length spectrum


%% file: Viterbo4.tex


\section{Proof of Theorem \ref{thm:Viterbo-sharper}}\label{sec:Viterbo}

In the course of the proof we first prove a similar theorem, where instead of the spectral norm, we consider the boundary depth \cite{UsherBD1,UsherBD2}.

\begin{thm}\label{thm:Viterbo-beta-sharper}
	Let $L \in \cl V$ be a Lagrangian submanifold of $M \in \cl W.$ Let $L' \subset D^*L$ be an exact Lagrangian submanifold of $D^*L$ that is an exact Lagrangian deformation of the zero section $L,$ considered as a Lagrangian submanifold of $M.$ That is, there exists a Hamiltonian isotopy $\{\phi^t\}$ of $M$ with $\phi^1(L) = L' \subset D^*L \subset M.$ Then for $c=c_L$ \begin{equation}\label{eq: beta bound Viterbo}\beta(L,L';\bK) \leq \frac{c}{2(1-c)}.\end{equation}
\end{thm}

\begin{rmk}\label{rmk:gamma sharpness S^1} We discuss sharpness for the other upper bounds \eqref{eq: beta bound Viterbo},\eqref{eq: gamma bound Viterbo}.
	
\begin{enumerate}
\item In the case $L=S^1,$ the bound \eqref{eq: beta bound Viterbo} is sharp by \cite[Lemma 45]{KS-bounds}. It takes the form \[{\beta(L,L') \leq 1/2 = C(S^1,g_{st})}.\] 	
	
\item The bound on $\gamma$ in case of $S^1$ with the round metric as above is $\gamma(L,L') \leq 3/4.$ We believe that it is not sharp, and expect the sharp bound to be $1/2,$ however our methods of producing upper bounds do not allow us to prove it. 
\item Applying \cite[Lemma 45]{KS-bounds} to ball embeddings into $D^*L$ relative to $L,$ for $L\in \mathcal{V}$ of dimension $\dim L > 1$, yields a lower bound of at most $\frac{1}{4}$ on \[\overline{\beta}(L,D^*L) = \sup_{\phi \in \Ham_c(D^*L)} \beta(L,\phi(L)),\] while the upper bound $\frac{c}{2(1-c)}$ equals $\frac{1}{2}$ for $L = S^n,$ and $\frac{n}{2}$ for other $L \in \cl V.$ The upper bound on \[\overline{\gamma}(L,D^*L) = \sup_{\phi \in \Ham_c(D^*L)} \gamma(L,\phi(L)),\] is weaker still. We were not able to improve upon this gap.
	\end{enumerate}

\end{rmk}

\input{multspec.tex}

\input{deform2.tex}

\input{roots.tex}

\input{neckstr.tex}

\input{Viterbo-wrapup.tex}

%% file: multspec.tex

\subsection{Multiplication operators and the bar-length spectrum}\label{sec:mult op chain level}

In this section we reinterpret certain differences of spectral invariants, that we call the {\em multiplication spectra}, in terms of bar-lengths of certain filtered complexes. These differences are conveniently defined as non-Archimedean spectral values of certain multiplication operators on the Lagrangian quantum homology. We will assume that $L \subset M$ is a wide monotone Lagrangian submanifold with $N_L \geq 2.$ Unless otherwise specified, we work with the Novikov field $\Lambda = \Lambda_{\tmon},$ while one could work equally well with $\Lambda = \Lambda_{\tmin}.$
%


Consider an element $a \in QH(L) = QH(L,\Lambda),$ with $\cA(a) \leq 0.$ It induces a multiplication operator \[m_a = (a \ast -): QH(L) \to QH(L).\] In particular, by the triangle inequality for spectral invariants from Section \ref{subsec:spec} we obtain the inequality \[c(x,\LH) - c(a \ast x, \LH) \geq 0 \] for all $x \in QH(L) \setminus \{0\}.$



Recall that given a Hamiltonian $H \in \cH,$ and perturbation datum $\cD,$ it follows from the properties of spectral invariants that \[l_{H;\cD} = c(-,H; \cD): QH(L) \to \R \cup \{-\infty\}\] is a non-Archimedean filtration. Considering the non-Archimedean spectral value decomposition \cite{UsherZhang} of $m_a: QH(L) \to QH(L)$ with respect to the non-Archimedean filtration function $l_{H;\cD},$ we obtain the {\em multiplication spectrum $a$}, which is given by \[0 \leq \beta_1(a,\LH; \cD) \leq \ldots \leq \beta_{B}(a,\LH; \cD),\] where \[B = \dim_{\Lambda_{\tmon}} QH(L) = \dim_{\bK} H_*(L; \bK),\] \[0 \leq \beta_1(a,\LH; \cD) \leq \ldots \leq \beta_{r(a)}(a,\LH; \cD) < +\infty\] are finite, and \[\beta_{r(a) + 1}(a,\LH; \cD) = + \infty, \ldots, \beta_{B}(a,\LH; \cD) = +\infty,\] where \[r(a) = \rank(m_a).\]

The absolute case, when we consider the Hamiltonian spectral invariants associated to $H \in \cH,$ on a closed symplectic manifold $(M,\omega),$ corresponds by Sections \ref{subsec:relations} and \ref{subsec:QH} to the case of the Lagrangian diagonal $M \cong \Delta_M \subset M \times M^{-},$ with a suitable Hamiltonian perturbation. We shall use the same notations as in the relative case to denote the multiplication spectrum of $a \in QH(M)$ with $\cA(a) \leq 0$ with respect to the non-Archimedean filtration \[l_{H;\cD}=c(-,H;\cD): QH(M) \to \R \cup \{-\infty\}.\]



The main result of this section is the following interpretation of the bar-length spectrum, as well as the multiplication spectrum of $a.$ Recall that a chain representative of $a$ with filtration level $\cA(a),$ in the chain complex computing $QH(L),$ induces a multiplication operator \[\mu_a = \mu_2(a, -) : CF(\LH;\cD,\Lambda) \to  CF(\LH;\cD,\Lambda).\] Each two such chain representatives give filtered chain-homotopic maps, hence all invariants considered below shall not depend on this choice. For a non-negative real number $\sigma \geq 0,$ we consider the complex \[Cone_\sigma(a,\LH; \cD) = Cone \Big(T^{\sigma} \cdot \mu_a: CF(\LH;\cD,\Lambda) \to  CF(\LH;\cD,\Lambda) \Big).\]


\begin{prop}\label{prop: mult spectrum via T^sigma}
For all $\sigma$ sufficiently large, the bar-length spectrum of $Cone_{\sigma}(a,\LH;\cD)$ is given by the $2K+r(a)$ finite lengths \[\beta_1(\LH; \cD) \leq \beta_1(\LH; \cD) \leq \ldots \leq \beta_K(\LH; \cD) \leq \beta_K(\LH; \cD) \leq \] \[ \leq \sigma + \beta_1(a,\LH; \cD) \leq \ldots \leq \sigma + \beta_{r(a)}(a,\LH; \cD),\] and precisely $2(B-r(a))$ infinite lengths. 
\end{prop}

The proof of this statement appears after that of Proposition \ref{prop: low-high spectrum separation} below.

%% file: deform2.tex

\subsection{A quantitative deformation argument for bar-length spectra}\label{sec:deformation}

In this section we discuss the effect of a deformation $d = d_0 + M',$ of a differential on a filtered $\Lambda$-vector space on the bar-length spectrum, assuming essentially that the valuation of $M'$ is sufficiently large relative to the barcode of one of $d_0, d$. We think of the complex $Cone_{\sigma}(a,\LH;\cD)$ above, considered once in $U = D^* L,$ and once in $M,$ where $U \subset M$ is a Weinstein neighborhood of $L,$ and $H^*\ul L,\cD$ are required to be supported in $U.$ 

Let $C$ be a filtered $\Lambda$-vector space, with filtration function $\cA:C \to \R \cup \{-\infty\}.$ Let $L:C \to C$ be a filtered $\Lambda$-linear map, that is for all $x \in C,$ $\cA(L(x)) \leq \cA(x).$ We set \[\cA(L) = \inf_{v \in C}  \Big( \cA(v) - \cA(L(v))  \Big).\] Clearly $\upsilon=\cA(L)$ is the maximal non-negative number such that there exists another filtered $\Lambda$-linear map $L_0:C \to C$ such that \[L = T^{\upsilon} L_0.\] 




Finally, recall from Section \ref{subsec:spec}, that for a filtered complex $(C,d)$ over $\Lambda$ with filtration function $\cA,$ the {\em induced filtration} \[\rH(\cA):H(C,d) \to \R \cup \{-\infty\}\] on the homology $\Lambda$-module $H(C,d)$ is defined by \[\rH(\cA)(a) = \inf_{[x] = a} \cA(x).\] It is a non-Archimedean filtration, still.



\begin{prop}\label{prop: low-high spectrum separation}
Let $(C,\cA)$ be a filtered finite-dimensional $\Lambda$-vector space and $(C,d)$ be a filtered $\Lambda$-chain complex structure on $(C,\cA).$ Let $\sigma > 0$ be a large parameter. Consider a filtered chain map $D:(C,d) \to (C,d).$ Consider the corresponding $\sigma$-shifted cone, \[\overline{C}_{\sigma} =  Cone(C,d,T^{\sigma} D).\] Then for all $\sigma$ sufficiently large, the bar-length spectrum of $\overline{C}_{\sigma}$ separates into two subspectra:

\begin{itemize}
\item{low:}  \[\beta'_1(\overline{C}_{\sigma})\leq \ldots  \leq \beta'_{K'}(\overline{C}_{\sig}) \; \ll  \sigma,\] 
\item{high:} \[\sigma < \;  \beta''_1(\overline{C}_{\sig}) \leq \ldots  \leq \beta''_{K''}(\overline{C}_{\sig}),\]
\end{itemize}
such that the low subspectrum is independent of $\sigma,$ and so is the high one, up to a shift by $\sigma.$ 

More precisely, for $\sigma \gg \beta(C,d),$ the low subspectrum $\beta'_1(\overline{C}_{\sigma}) \leq \ldots \leq \beta'_{K'}(\overline{C}_{\sigma})$ coicides with the doubled bar-length spectrum $\beta_1(C,d) \leq \beta_1(C,d) \leq \ldots \leq \beta_K(C,d) \leq \beta_K(C,d)$ of $(C,d).$ In particular $K' = 2K.$ The high bar-length spectrum statisfies \[\beta''_k(\overline{C}_{\sigma}) = \sigma + \beta_k(Cone(H_*(C,d),[D])),\] for the cone of the map on homology: $[D]: (H_*(C,d),\rH(\cA)) \to (H_*(C,d),\rH(\cA)).$


\end{prop}

This statement can be proven rather conceptually by considerations involving cones in the derived category of $\pemod$ \cite{PolShe-conv}. However, to prepare for further chain-level deformation arguments, we provide a different chain-level proof.


\begin{proof}[Proof of Proposition \ref{prop: low-high spectrum separation}]

We argue as follows. Extend coefficients to $\Lambda_{\tuniv},$ and abbreviate $\Lambda = \Lambda_{\tuniv}.$ Now let $E = \{x_1,\ldots,x_B,y_1,\ldots, y_K,z_1,\ldots,z_K\},$ be an orthonormal basis of $H(C,d)$ over $\Lambda,$ with $\ker(d) = \Lambda\left<\{x_1,\ldots,x_B,z_1,\ldots,z_K\}\right>$ and $d(y_j) = T^{\beta_j} z_j.$ It is convenient to denote $X = \Lambda\left<\{x_1,\ldots,x_B\}\right>,$ $Y = \Lambda\left<\{y_1,\ldots,y_K\}\right>,$ $Z = \Lambda\left<\{z_1,\ldots,z_K\}\right>,$ and let $\pi_X: C \to X$ be the projection onto $X$ along $Y \oplus Z.$ By orthogonality, it is immediate to check the following fact.

\begin{lma}\label{lemma:orthogonal projection}
For all $c \in C,$ $\cA(\pi_X(c)) \leq \cA(c).$ 
\end{lma}  

As a consequence, we obtain the following.

\begin{lma}\label{lemma:orthogonal basis for induced filtration}
The $\Lambda$-basis $\rH(E) = \{[x_1],\ldots,[x_B]\}$ of $\rH(C,d)$ is orthogonal with respect to the induced filtration $\rH(\cA),$ and the matrix $P$ of $[D]$ in the basis $\rH(E)$ coincides with that of $\pi_X \circ D|_X : X \to X$ in basis $\{x_1,\ldots,x_B\}.$
\end{lma}

Indeed, the statement on matrices is evident, while the first statement follows from Lemma \ref{lemma:orthogonal projection} and the relation $\pi_X(c) = x$ for $x \in X$ and $c \in \ker(d)$ satisfying $[x] = [c].$

Therefore it remains to show that for $\sigma \gg \beta(C,d),$ $(\overline{C},\overline{d}_{\sigma})$ admits orthonormal bases $\overline{E}_1, \overline{E}_2$ of $\overline{C}$ in which the matrix $[\overline{d}_{\sigma}]^{\overline{E}_1}_{\overline{E}_2}$ takes the block form with two diagonal blocks $\delta = \mathrm{diag}(T^{\beta_1},\ldots,T^{\beta_K})$ and one block $T^{\sigma} P,$ such that no two blocks share rows or columns. This is immediate from the assumption $\sigma \gg \beta(C,d)$ by applying elementary operations over $\Lambda_0$ to the rows and columns of the matrix $[\overline{d}_{\sigma}]^{\overline{E}}_{\overline{E}}$ of $\overline{d}_{\sigma}$ in the orthonormal basis $\overline{E} = E \times \{0\} \cup \{0\} \times E$ of $\overline{C}.$ (This corresponds to multiplying the latter matrix by matrices in $GL(\dim_{\Lambda}(C),\Lambda_0)$ on the left and on the right, and we recall that multiplication by such matrices sends orthonormal bases to orthonormal bases.) For future use, we let $\overline{E}'_1, \overline{E}'_2$ be orthonormal bases of $\overline{C}$ for which the block form is of two diagonal blocks $\delta,$ and one diagonal block $T^{\sigma} P',$ where $P'$ is the Smith normal form over $\Lambda_0$ of $P.$

\end{proof}

\begin{proof}[Proof of Proposition \ref{prop: mult spectrum via T^sigma}]
	This is an immediate consequence of Proposition \ref{prop: low-high spectrum separation}, once we observe the following. The map $PSS_{\LH;\cD}: QH(L) \to HF(\LH;\cD)$ is an isomorphism of $\Lambda$-modules that intertwines the maps $m_a: QH(L) \to QH(L)$ and $[\mu_a]: HF(\LH; \cD) \to HF(\LH;\cD).$ Moreover, from the definitions, is clear that the filtration functions $l_{\LH;\cD}$ and $\rH(\cA_{\LH;\cD})$ are related by \[l_{\LH;\cD} = (PSS_{\LH;\cD})^* \rH(\cA_{\LH;\cD}).\]
\end{proof}


We turn to our first deformation argument.

\begin{prop}\label{prop: deformation no sigma}
Let $(C,\cA)$ be a filtered finite-dimensional $\Lambda$-vector space and $(C,d_0),(C,d)$ be two filtered $\Lambda$-chain complex structures on $(C,\cA).$ Write $d = d_0 + M.$ Let $A > 0$ be a positive number. Assume that $\cA(M) \geq A.$ Then the bar-length spectra of $(C,d_0),$ $(C,d)$ below $A$ coincide. That is if $\beta_{k}(C,d) < A$ then $\beta_k(C,d_0) = \beta(C,d),$ and vice versa.
\end{prop}	

\begin{proof}

We give a proof that shall be extensively used in the proof of Case \ref{case:equality for high} of Proposition \ref{prop: deformation}. It is given by a matrix calculation that is similar to the one in the proof of Proposition \ref{prop: low-high spectrum separation}.

Consider the matrix of $d$ in the normal form orthonormal basis $E = \{x_j,y_j,z_j\}$ as above. It is composed of one block \[\delta = \mathrm{diag}(T^{\beta_1},\ldots,T^{\beta_K}),\] with $\beta_j = \beta_j(C,d)$ the bar-length spectrum of $(C,d),$ and all other blocks are zero. Let $1 \leq l \leq K$ be the index for which $\beta_k < A$ for all $k \leq l,$ and $\beta_k \geq A$ for all $k > l.$ We subdivide $\delta$ into two blocks $\delta_l = \mathrm{diag}(T^{\beta_1},\ldots,T^{\beta_l})$ and $\delta_{+} = \mathrm{diag}(T^{\beta_{l+1}},\ldots,T^{\beta_K}).$ Let the block matrix $\overline{\delta}_l$ consist of the block $\delta_l$ extended by zero blocks. Now write $M = T^A M'$ with $\cA(M') \geq 0.$ The matrix of $d_0 = d - T^A M'$ in the basis $E$ is of the form $\overline{\delta}_l + T^A M_1,$ for a matrix $M_1$ with $\Lambda_0$ coefficients. Performing row and column elementary operations over $\Lambda_0$ we obtain a block matrix form, with one block being $\delta_l$ and not sharing rows or columns with other blocks, while all other blocks have coefficients in $T^A \Lambda_0.$ This implies that the bar-length spectrum of $(C,d_0)$ below $A$ is given by $\beta_1,\ldots,\beta_l.$

\end{proof}

We proceed with our main algebraic deformation argument.

\begin{prop}\label{prop: deformation}

Let $(C,\cA)$ be a filtered finite-dimensional $\Lambda$-vector space and $(C,d_0),(C,d)$ be two filtered $\Lambda$-chain complex structures on $(C,\cA).$ Write $d = d_0 + M'.$ Let $\sigma > 0$ be a large parameter, and $A>a>0$ be positive numbers. Consider two filtered chain maps $D:(C,d) \to (C,d),$ $D_0:(C,d_0) \to (C,d_0).$ Write $D = D_0 + N'.$ We consider the corresponding $\sigma$-shifted cones, \[\overline{C}_{0,\sigma} = Cone(C,d_0,T^{\sigma} D_0),\] \[\overline{C}_{\sigma} =  Cone(C,d,T^{\sigma} D).\] Assume that \[\cA(M') \geq a,\] \[\cA(N') \geq A.\] Then for all $\sigma$ sufficiently large:

\begin{enumerate}
\item \label{case:equality for low} If \[\beta'_{l}(\overline{C}_{\sigma}) < a,\] for some $l\geq 1,$ then the low bar-length spectra of $\overline{C}_{0,\sigma},\overline{C}_{\sig}$ satisfy \[ \beta'_k(\overline{C}_{0,\sigma}) = \beta'_k(\overline{C}_{\sigma})\] for all $1 \leq k \leq l.$
\item \label{case:equality for high} Assume that $H(C,d) \cong H(C,d_0),$ so that $K'_0 = K',$ for the low bar-length spectra of $\overline{C}_{0,\sig},\overline{C}_{\sig},$ and that $\beta'_{K'}(\overline{C}_{\sigma}) < a.$ If moreover \[\beta''_{l}(\overline{C}_{\sigma}) < \sigma + A-a,\] for some $l\geq 1,$ then the high bar-length spectra of $\overline{C}_{0,\sig},\overline{C}_{\sig}$ satisfy \[ \beta''_k(\overline{C}_{0,\sig}) = \beta''_k(\overline{C}_\sig)\] for all $1 \leq k \leq {l}.$
\end{enumerate}




\end{prop}

\begin{proof}[Proof of Proposition \ref{prop: deformation}]

%
%
Case \ref{case:equality for low} is an immediate consequence of Propositions \ref{prop: low-high spectrum separation} and \ref{prop: deformation no sigma}. We turn to the proof of Case \ref{case:equality for high}. Consider the matrix $[\overline{d}_{0,\sigma}]$ of $\overline{d}_{0,\sigma}$ in the base $\overline{E}$. 


It takes the following form, where we denote by $\ast$ a matrix with coefficients in $\Lambda_0:$
\begin{equation}\label{matrix 1 of d_0 sigma}  
\begin{pmatrix}  
T^A \ast& T^A \ast& T^A \ast           & T^{\sigma} P + T^{\sigma + A}\ast  & T^{\sigma} \ast & T^{\sigma+A} \ast\\ 
T^A \ast& T^A \ast& T^A \ast           & T^{\sigma+A} \ast & T^{\sigma} \ast & T^{\sigma+A} \ast\\ 
T^A \ast &\delta + T^A \ast & T^A \ast & T^{\sigma} \ast  & T^{\sigma} \ast & T^{\sigma} \ast\\
0 & 0 & 0                              & T^A \ast & T^A \ast& T^A \ast \\
0 & 0 & 0                              & T^A \ast & T^A \ast& T^A \ast \\
0 & 0 & 0                              & T^A \ast & \delta +T^A \ast& T^A \ast \\
\end{pmatrix}\end{equation}

Clearing the sixth row and then the fifth column by elementary operations over $\Lambda_0$ by means of the block $\delta + T^A \ast,$ and using the assumtion $H(C,d_0) \cong H(C,d),$ as well as the condition that $\delta = \diag(T^{\beta_1},\ldots,T^{\beta_K})$ with $\beta_j \leq a$ for all $1 \leq j \leq K,$ we obtain the following new block form:
\begin{equation}\label{matrix 2 of d_0 sigma}  
\begin{pmatrix}  
T^A \ast& T^A \ast& T^A \ast           & T^{\sigma} P + T^{\sigma + A-a}\ast  & 0 & T^{\sigma+A-a} \ast\\ 
T^A \ast& T^A \ast& T^A \ast           & T^{\sigma+A-a} \ast & 0 & T^{\sigma+A-a} \ast\\ 
T^A \ast &\delta + T^A \ast & T^A \ast & T^{\sigma} \ast  & 0 & T^{\sigma} \ast\\
0 & 0 & 0                              & 0 & 0 & 0 \\
0 & 0 & 0                              & 0 & 0 & 0 \\
0 & 0 & 0                              & 0  & \delta & 0 \\
\end{pmatrix}\end{equation}

Similarly, clearing the third row and then the second column by elementary operations over $\Lambda_0$ by means of the block $\delta + T^A \ast,$ and using the assumtion $H(C,d_0) \cong H(C,d),$ as well as the condition $\beta_j \leq a$ for the exponents of $\delta,$ we obtain the block form:
\begin{equation}\label{matrix 3 of d_0 sigma}  
\begin{pmatrix}  
0& 0& 0           & T^{\sigma} P + T^{\sigma + A-a}\ast  & 0 & T^{\sigma+A-a} \ast\\ 
0& 0& 0           & T^{\sigma+A-a} \ast & 0 & T^{\sigma+A-a} \ast\\ 
0 &\delta  & 0 & 0  & 0 & 0\\
0 & 0 & 0                              & 0 & 0 & 0 \\
0 & 0 & 0                              & 0 & 0 & 0 \\
0 & 0 & 0                              & 0  & \delta & 0 \\
\end{pmatrix}\end{equation}

Now as in the second proof of Proposition \ref{prop: deformation no sigma}, recalling that the Smith normal form of $T^{\sigma} P$ over $\Lambda_0$ is given precisely by 
$\diag(T^{\beta''_1},\ldots, T^{\beta''_{K''}}),$ and separating $\Delta_l = \diag(T^{\beta''_1},\ldots, T^{\beta''_{l}})$ for the index $l$ with $\beta''_k < \sigma + A-a$ if and only if $k \leq l,$ and performing row and column operations we obtain the block form
\begin{equation}\label{matrix 4 of d_0 sigma}  
\begin{pmatrix}  
0& 0& 0& \Delta_l & 0 & 0 & 0\\
0& 0& 0           &  0&T^{\sigma + A-a}\ast  & 0 & T^{\sigma+A-a} \ast\\ 
0& 0& 0           & 0& T^{\sigma+A-a} \ast & 0 & T^{\sigma+A-a} \ast\\ 
0 &\delta  & 0 &  0& 0 & 0 & 0\\
0 & 0 & 0                            &0  & 0 & 0 & 0 \\
0 & 0 & 0                            &0  & 0 & 0 & 0 \\
0 & 0 & 0                             &0 & 0  & \delta & 0 \\
\end{pmatrix}\end{equation}

This finishes the proof.

\end{proof}

\begin{rmk}
We note that Propositions \ref{prop: low-high spectrum separation} and \ref{prop: deformation} have evident analogues for maps $D:(C,d) \to (C',d'),$ where the domain and target complexes are not necessarily the same. These versions are proven in the same way, and we expect them to be useful in the future.
\end{rmk}

%% file: roots.tex
\subsection{Roots of unity and bounds on multiplication spectra}\label{sec:roots}

In this section we prove that if $a \in QH_{n-k}(L)$ with $\cA(a) \leq 0,$ is a {\em quantum root of unity} in the sense that \[a^{m} = t^{-mk} [L]\] for some $m \geq 2,$ then the multiplication spectrum of $a$ satisfies certain natural bounds. While this is not necessary in this particular section, for future use we assume that $k < N_L.$ In view of \cite[Section 3]{KS-bounds} the Lagrangian submanifolds $L \subset M$ for $L \in \cl V$ and its respective $M \in \cl W,$ the point class $[pt] \in QH_0(L)$ is a quantum root of unity, with $k = \dim L.$

\begin{prop}\label{prop:root of unity bound}
Let $a \in QH_{n-k}(L)$ be a quantum root of unity. Then $r(a) = B,$ and the following estimates hold for all $H \in \cH$ and all perturbation data $\cD:$
\begin{enumerate}
\item \label{root of unity bound: easy case} $\beta_B(a,\LH;\cD) \leq \frac{mk}{N_L} A_L.$
\item \label{root of unity bound: difficult case} $\beta_{\lceil B/m \rceil}(a,\LH;\cD) \leq \frac{k}{N_L} A_L.$
\end{enumerate}
\end{prop}

\begin{proof}[Proof of Proposition \ref{prop:root of unity bound}]
	

For Case \eqref{root of unity bound: easy case}, we argue as follows. Given $x \in QH(L) \setminus \{0\},$ formally setting $a^0 = [L],$ we calculate \begin{equation} \label{eq: sum of mult differences} \sum_{j=0}^{m-1} c(a^j x, H; \cD) - c(a^{j+1}x, H;\cD) = \frac{mk}{N_L} A_L,\end{equation} and since each summand \[c(a^j x, H; \cD) - c(a^{j+1}x, H;\cD) \geq 0\] is non-negative, we immediately obtain the statement.

Case \eqref{root of unity bound: difficult case} requires a slightly more delicate argument. Denote $V = QH(L),$ and $L = m_a$ considered as a linear operator $L:V \to V.$ If $\beta_{\lceil B/m \rceil}(a,H;\cD) > \frac{k}{N_L} A_L,$ then there exists a $\Lambda$-linear subspace $W \subset V$ of dimension \begin{equation} \label{eq: dim big subspace W is large} \dim_{\Lambda} W > \frac{m-1}{m} B\end{equation} such that for all $w \in W,$ \[c(w,H;\cD) - c(L(w),H;\cD) > \frac{k}{N_L} A_L.\] By \eqref{eq: dim big subspace W is large} the subspace $I = W \cap L(W) \cap \ldots \cap L^{m-1} W,$ satisfies \[\codim (I) \leq m \cdot \codim(W) = m(B - \dim W) < B.\] Therefore $I$ is not trivial, and taking $x \in I \setminus \{0\},$ we obtain \[\sum_{j=0}^{m-1} c(a^j x, H; \cD) - c(a^{j+1}x, H;\cD) > m\cdot \frac{k}{N_L} A_L,\] in contradiction to \eqref{eq: sum of mult differences}.

\end{proof}

%% file: neckstr.tex

\subsection{From geometry to algebra: SFT and neck-stretching}\label{sec:SFT}

{In this section we provide the geometric underpinning for applying the algebraic situation described above to comparing the Lagrangian Floer theory of two anchored brandes $(\ul L, \ul L')$ contained inside a Weinstein neighborhood $U \subset T^* L,$ as computed inside $U$ with Hamiltonian perturbation supported therein, and the Lagrangian Floer theory of $(\ul L, \ul L')$ as computed inside $M,$ when $U$ is symplectically embedded into $M,$ with the perturbation extended naturally thereto. 

This necessitates a careful choice of an almost complex structure, which in particular, makes a certain symplectic divisor (symplectic submanifold of real codimension $2$) in the complement of $L$ complex. For Donaldson divisors in the complement of a Lagrangian we refer to the early paper \cite{AurouxGayetMohsen} (cf. \cite{Verine-Don}) for existence, and for associated Floer theoretic considerations to \cite{CharestWoodward-Floer17,CharestWoodward-Fukaya}. Further, we shall perform the operation of stretching the neck, and invoke compactness results in SFT (symplectic field theory). We refer to \cite{BEHWZ-compactness, CM-compactness, CM-Lagr} for these topics.}

We start with a basic uniform energy bound for the Floer complex of {$(\LH;\cD)$} in $M,$ together with the multiplication operator by $a =[pt] \in QH(L).$ Of course this naturally produces a corresponding uniform energy bound for the (isomorphic up to uniform shift) Lagrangian Floer complex for $(\ul L, H^* \ul L; \cD)$ with multiplication operator by $a = [pt] \in QH(L).$ We note that it uses monotonicity rather strongly.


\begin{prop}\label{prop: uniform energy bound}
Let $m \in \Z_{\geq 0}.$ There exists a constant $E_0$ depending only on $m,(L,M,H^{\cD}),$ such that all the solutions $u$ of index $\ind(u) \leq m$ of the Floer equation with respect to $(H^{\cD},L)$ and an $\om$-compatible almost complex structure $J \in \cJ(M,\om),$ with asymptotic conditions at $(H^{\cD},L)$-chords satisfy $E(u) \leq E_0.$
\end{prop}


\begin{proof}
This is an elementary consequence of the monotonicity condition. For each generator $(x,\bar{x}) \in \til{\cO}(\LH;\cD)$ associate the reduced action by \[\hat{\cA}(x) = \cA(x,\bar{x}) - \kappa \cdot \mu_{CZ}(x,\bar{x}).\] As notation suggests, by the monotonicity condition $\hat{\cA}$ depends only on the corresponding chord $x \in \cO(\LH;\cD).$ Now for a Floer trajectory $u$ from $x \in \cO(\LH;\cD)$ to $y \in \cO(\LH;\cD),$ we obtain the identity \[ \hat{\cA}(x) - \hat{\cA}(y) = E(u) - \kappa \ind(u).\] Hence for $\ind(u) \leq m,$ we have the bound \[E(u) \leq C_0,\] for \[C_0 = \kappa m + \max\{\hat{\cA}(x) - \hat{\cA}(y)\,|\, x, y \in \cO(\LH;\cD) \}.\]
\end{proof}





As $L \subset M$ is {\em homologically} monotone, by \cite{CharestWoodward-Floer17,CharestWoodward-Fukaya,Giroux-Don,Verine-Don} there exists a Donaldson divisor $\Sigma,$ with $[\Sigma] = k \cdot PD([\omega]),$ for a certain $k>0,$ such that $L \subset M \setminus \Sigma$ is {\em exact}, and moreover for each $\epsilon_0 >0,$ $\Sigma$ can be arranged to satisfy $L \subset M \setminus D_{(1-\epsilon_0)/k} \Sigma,$ where $M = \mathrm{Skel}_{\Sigma} \sqcup D_{1/k} \Sigma$ is the corresponding Biran decomposition \cite{Biran:Nonintersections,Biran:Barriers}, and $D_{r_0/k}$ for $r_0 \in [0,1]$ is the disk subbundle of $D_{1/k} \Sigma,$ whose fiber is a standard symplectic disk of area $r_0/k.$ In fact, by \cite{Verine-Don}, $L$ can be made to be a subset of $\mathrm{Skel}_{\Sigma}.$ We assume the latter situation, because it holds in all the examples we require, however only the above property on $\epsilon_0$ is actually necessary for the proof. For each $r_0 \in (0,1)$ let $U_{r_0}$ be a Weinstein neighborhood of $L,$ symplectomorphic to a co-disc cotangent bundle, whose closure is contained in $M \setminus D_{(1-r_0)/k} \Sigma.$ We proceed to perform neck stretching around $\partial U_{r_0},$ referring for the details of the construction to \cite{BEHWZ-compactness,CM-compactness}. We also refer to \cite{MembrezOpshtein} where neck-stretching in a similar setup was applied to study questions of $C^0$-symplectic topology. This procedure produces almost complex structures $J_{\tau} \in \cJ(M,\om)$ for $\tau \in  \Z_{>0},$ deforming a given almost complex structure in a collar neighborhood of $\partial U_{r_0},$ such that $J_{\tau}$ holomorphic objects, as $\tau \to \infty,$ satisfy a far-reaching generalization of Gromov compactness. We claim the following.

\begin{prop}\label{prop: stretched complex first}
For sufficiently large $\tau \in \Z,$ solutions $u$ as in Proposition \ref{prop: uniform energy bound}, with $J = J_{\tau} \in \cJ_M,$ either lie in $ U_{r_0},$ or intersect $\Sigma,$ and have energy $E(u) > (u \circ \Sigma) \cdot A',$ where $A' = r_0/k.$
\end{prop}


%

\begin{proof}

Assume that $u_l$ is a $J_{\tau_l}$-Floer trajectory, that does not intersect $\Sigma,$ where $\tau_l \to \infty,$ and is not contained in $U_{r_0}$ for all $l.$ Then $u_l$ converges to a pseudo-holomorphic building as in the papers \cite{BEHWZ-compactness,CM-compactness}, which adapt perfectly well to the situation of holomorphic curves with boundary on Lagrangian submanifolds, perhaps with a Hamiltonian perturbation term (that vanishes near the hypersurface along which the neck is stretched), with the property that the lowest level contains a unique component, which we call the root, with two boundary punctures asymptotic to intersection points $x,y$ of $L$ and $L',$ and perhaps other components. Topologically, all components of the building glue to an index $\leq m$ relative homotopy class of Whitney disks $[(\bb D, (\partial \bb D)^{-},(\partial \bb D)^{+}, \{-1\},\{1\}),(M, L, L', {x},{y})],$ where $(\partial \bb D)^{\pm} = \partial \bb D \cap \{\pm \Im(z) \geq 0\},$ and have positive $\om$-area.  Finally, there is at least one component in a level higher than the lowest. In other words there exits a {\em non-empty} collection of Reeb orbits $\{\gamma_j\}$ in $\partial U_{r_0}$ entering into the building, to which the root component is asymptotic at interior punctures. 


Let $C_j$ be the topological disks obtained as follows. Glue, topologically, the components of the building lying in the complement of the root component (this necessarily includes some higher level components). Then $C_j$ is the connected component of the resulting surface corresponding to $\gamma_j.$ Further gluing to $C_j$ a trivial cylinder over $\gamma_j,$ whose symplectic area is the period of $\gamma_j,$ we obtain a topological disk $\hat{C}_j$ in $M \setminus \Sigma,$ with boundary on $L,$ with $\int_{\hat{C}_j} \omega > 0.$ However, as $L$ is exact in $M \setminus \Sigma,$ we must have $\int_{\hat{C}_j} \omega = 0.$ This is a contradiction implying that for $\tau$ sufficiently large, if $u$ does not intersect $\Sigma$ then it lies inside $U_{r_0}.$

It remains to prove that if $u$ does intersect $\Sigma,$ then $E(u) > (u \circ \Sigma) \cdot A',$ where $A' = r_0/k.$ This is done either by another neck-stretching near ${r = \epsilon''/k} \sqcup {r = (r_0 + \epsilon'')/k}$ or by the following choice of $J$ in $D_{(r_0+\epsilon'')/k} \Sigma,$ where $\epsilon'' > 0$ is sufficiently small: we require the projection $\pi_{\Sigma}: D_{(r_0+\epsilon'')/k} \Sigma \to \Sigma$ to be $(J,J_{\Sigma})$-holomorphic, for a fixed $J_{\Sigma} \in \cJ(\Sigma).$ We leave the first approach to the interested reader, and describe the second one.

In that case, the symplectic structure in the disk bundle $D\Sigma$ is given by $\omega = \pi_{\Sigma}^* \omega_{\Sigma} + d(r \theta),$ where $\theta$ is the prequatization form on the normal principal $S^1$-bundle $P \to \Sigma,$ such that $\pi_{\Sigma}^* \omega_{\Sigma} = - d\theta.$  Note that $D_{(r_0+\epsilon'')/k}\Sigma = P \times_{S^1} \mathbb{D}((r_0+\epsilon'')/k),$ where $\mathbb{D}((r_0+\epsilon'')/k)$ is the standard disk of capacity $(r_0+\epsilon'')/k.$ Let $v$ be a two-cycle with boundary on $D_{(r_0+\epsilon'')/k} \Sigma,$ the circle bundle given by $\{r = (r_0+\epsilon'')/k\}.$ It is well-known (see \cite{Frauenfelder}), that $v \circ \Sigma$ is given as $\int_{\partial v} \theta + \int_{v} \pi_{\Sigma}^* \omega_{\Sigma}.$ Hence for $v$ being $J$-holomorphic, where $J \in \cJ(M,\omega)$ satisfies the above property, and perhaps decreasing $\epsilon''$ to $0<\epsilon' \leq \epsilon''$ to make the intersection with $\{r = (r_0 + \epsilon')/k\}$ transverse, we obtain \[\int_{v} \omega =  (r_0/k + \epsilon') \cdot \int_{\partial v} \theta + \int_{v} \pi_{\Sigma}^* \omega_{\Sigma} > r_0/k \cdot (v \circ \Sigma).\]

\end{proof}

%

\begin{rmk}\label{rmk: stretched complex - actual lower bound on energy}
It is important to remark that since $L$ is exact outside $\Sigma,$ and $[\Sigma] = PD(k[\omega]),$ where we assume that $[\omega]$ is a rational class with denominator dividing $k,$ by \cite[Lemma 3.4]{CharestWoodward-Floer17}, as well as \cite[Section 3.3]{TonkonogDescendants}, the intersection number $u \circ \Sigma$ is an integer multiple of $k A_L.$ Hence, by positivity of intersections \cite{CieliebakMohnkePosInt} we obtain that \[E(u) \geq A = r_0 \cdot A_L,\] if $u$ intersects $\Sigma$ in the setting of Proposition \ref{prop: stretched complex first}.
\end{rmk}

\begin{rmk}
	The approach of the adapted choice of the almost complex structure was developed, for a different purpose, in discussions with D. Tonkonog and R. Vianna in the course of preparation of \cite{STV-Flux18}.
\end{rmk}




By applying Gromov compactness, it is easy to pass $J^{(\LH)}$ sufficiently close to $J_{\tau},$ with $\tau \gg 1,$ for which transversality for $CF(\LH;\cD),$ as well as for the action $QH(L) \otimes CF(\LH;\cD) \to CF(\LH;\cD)$ holds, and the same properties as in Proposition \ref{prop: stretched complex first} and Remark \ref{rmk: stretched complex - actual lower bound on energy} remain true for all Floer configurations involved. Note that this brings us to the general situation of the above two sections.

%

%% file: Viterbo-wrapup.tex


\subsection{Completing the proof}
Finally, we apply Sections \ref{sec:SFT} and \ref{sec:roots} to get the complex from Section \ref{sec:mult op chain level} into the situation of Section \ref{sec:deformation}, and inspect the outcome.


As above, we consider $L \subset M,$ $L \in \cl V,$ $M \in \cl W,$ as a monotone Lagrangian submanifold. Assume that $L' \subset D^*L$ is Hamiltonian isotopic to $L$ inside $M.$ By taking inverses, we obtain a Hamiltonian $H \in \cl H_M$ such that $L' = (\phi^1_H)^{-1} L.$ Pick an anchored Lagrangian brane $\ul L$ decorating $L,$ with $\alpha$ the constant path at a certain $x \in L,$ and $\lambda \equiv T_x L;$ it is supported in $U = D^*L = M \setminus \Sigma.$ Consider the Lagrangian brane $\ul L' = H^* \ul L$ decorating $(\phi^1_H)^{-1} L.$ The class $[pt] \in QH_0(L)$ is a quantum root of unity, with codegree $k = n_L = \dim(L) < N_L.$ Moreover, by \cite[Theorem G]{KS-bounds} we have $\beta(\underline{L},\underline{L}';\cD') = \beta(\LH;\cD) \leq c \cdot A_L$ where $0 < c = c_L = \frac{n_L}{N_L} < 1,$ for all $H \in \cH,$ and perturbation data $\cD.$ Moreover, Proposition \ref{prop:root of unity bound} implies that at least one of the bar-lengths in the multiplication spectrum of $m_{[pt]} : QH(L) \to QH(L)$ filtered by $l_{H;\cD} = l_{\underline{L},\underline{L}';\cD'}$ is at most $c \cdot A_L.$ Now, since $M$ is simply connected, the auxiliary data in $\underline{L}'$ is homotopic to one contained in $U,$  which we fix for the rest of the proof. The new homotopic data gives a filtered graded complex that is identical to the one of $(\ul L, \ul L')$ with the initial data, up to a uniform shift in the filtration. Spectral norms, and all bar-length spectra we consider are hence independent of this modification. Furthermore, we pick the perturbation datum $\cD'$ for $(\ul L, \ul L')$ supported in $U,$ so that outside the embedding of $U$ in $M$ the almost complex structure extends smoothly to $M,$ and the complex structure is of cylindrical SFT type near the boundary of $U.$ For $0<s\leq 1$ let $\underline{L}'_s = s \cdot \underline{L}'$ be the the image of $\underline{L}'$ under the scaling on $T^*L$ given by the scalar $\R_{>0}$-action on the fibers. Note that $\underline{L}'_s$ is supported in $s \cdot U,$ and, since $L' \subset D^*L$ is exact, $\underline{L}'_s = (K_s)^* \ul L',$ for a Hamiltonian $K_s \in \cH_{U}.$ Taking $s<1-c$ we obtain from Proposition \ref{prop: deformation no sigma} the identity
$\beta(\underline{L},\underline{L}'_s;s\cdot \cD',s\cdot U) =  \beta(\underline{L},\underline{L}'_s;s\cdot \cD',M) \leq c \cdot A_L.$ However an elementary rescaling argument for the Floer complex gives $\beta(\underline{L},\underline{L}'_s;s\cdot \cD',s\cdot U) = s \cdot \beta(\underline{L},\underline{L}';\cD',U),$ whence $s \cdot \beta(\underline{L},\underline{L}';\cD',U) \leq c \cdot A_L$ for all $s < 1-c.$ Therefore

\begin{equation}\label{eq:beta bound}
\beta(\underline{L},\underline{L}';\cD',U) \leq \frac{c}{1-c} \cdot A_L
\end{equation}

Now we choose $0<s<1$ so that \begin{equation}\label{eq:bs constraint}b(s) = (1-s) A_L - s \frac{c}{1-c} \cdot A_L > {c}\cdot A_L.\end{equation}Propositions  \ref{prop: uniform energy bound} and \ref{prop: deformation no sigma}  yield first that the bar-length spectra of $CF(\underline{L},\underline{L}'_s;s\cdot \cD',s\cdot U)$ and of $CF(\underline{L},\underline{L}'_s;s\cdot \cD',M)$ coincide. Proposition \ref{prop: mult spectrum via T^sigma} yields moreover that for all $\sigma$ sufficiently large \[\beta_1''(Cone_{\sigma}(\pt,\underline{L},\underline{L}'_s;s\cdot \cD',M)) - \sigma \leq c \cdot A_L,\] and that for $l=1$ the conditions of Proposition \ref{prop: deformation} are satisfied, whence we have \begin{equation}\label{eq:high bar-length spectrum out and in}\beta_1''(Cone_{\sigma}(\pt,\underline{L},\underline{L}'_s;s\cdot \cD',s\cdot U)) = \beta_1''(Cone_{\sigma}(\pt,\underline{L},\underline{L}'_s;s\cdot \cD',M))
\end{equation} and the low bar-length spectra of ${Cone_{\sigma}(\pt,\underline{L},\underline{L}'_s;s\cdot \cD',s\cdot U)}$ and ${Cone_{\sigma}(\pt,\underline{L},\underline{L}'_s;s\cdot \cD',M)}$ coincide. In particular \[\beta_1''(Cone_{\sigma}(\pt,\underline{L},\underline{L}'_s;s\cdot \cD',s\cdot U)) - \sigma \leq c \cdot A_L.\] Now by Proposition \ref{prop: mult spectrum via T^sigma} again, and an obvious homological calculation amounting to $[pt] \cap [L] = [pt],$ and $\ima(-\cap [L]) = \bK [pt]$ (see Section \ref{subsubsec:spec norm}) we have \[\gamma(\underline{L},\underline{L}'_s; s\cdot \cD', s \cdot U) = \beta_1((\underline{L},\underline{L}'_s; s\cdot \cD', s \cdot U), [pt]) \leq c \cdot A_L,\] and again by rescaling, we obtain \[\gamma(\underline{L},\underline{L}'; \cD', U) \leq c \cdot A_L/s,\]  under the constraint \eqref{eq:bs constraint} on $s.$  The infimimum of the right hand side under this constraint equals to its value at $s_\ast = \frac{1-c}{1+c},$ hence we obtain \[\gamma(\underline{L},\underline{L}'; \cD', U) \leq (1+c)\frac{c}{1-c} \cdot A_L.\] Now, by continuity of $\gamma,$ we may take the limit as the Hamiltonian term in $\cD'$ goes to zero, and note that by Section \ref{subsubsec:spec norm} we get: \[\gamma(L,L'; U) \leq (1+c)\frac{c}{1-c} \cdot A_L.\]

%% file: applications.tex
\section{Applications to $C^0$ symplectic topology, and quasi-morphisms}

%
%
%
%

\begin{proof}[Proof of Theorem \ref{thm:gamma is C^0 continuous on CP^n}] First we prove the following two lemmas.
	
\begin{lma}\label{lma: orthogonal basis}
Let $H \in \cH$ be a Hamiltonian on $(M,\omega) = (\C P^n, \om_{FS}),$ and let $a \in QH_{2n-2}(\C P^n, \Lambda)$ be the hyperplane class. Then $S = \{[M] = a^0, a^1,\ldots, a^n\}$ form an orthogonal $\Lambda$-basis of $QH_*(\C P^n, \Lambda)$ with respect to the non-Archimedean filtration function $l_{H}(-) = c(H,-).$ 
\end{lma}	 

\begin{proof}[Proof of Lemma \ref{lma: orthogonal basis}:]
We first choose coefficients $\Lambda = \Lambda_{\tmin.}$ Observe that $S$ is a $\Lambda$-basis of $QH_*(\C P^n, \Lambda).$ Consider a non-zero linear combination $w = \sum_{0 \leq j \leq n} \la_j a^j.$ We shall prove that \[c(H,w) = \max \{c(H,\lambda_j a^j)\} = \max \{c(H, a^j) - \nu(\lambda_j)\}.\] Consider first the case when $H$ is non-degenerate. Then it is easy to see that one case choose a Floer perturbation datum $\cD = (J^H,K^H)$ with $K^H$ being as $C^2$-small as necessary, such that $H^{\cD} = H \# K^H$ is also non-degenerate, and moreover for each two distinct contractible periodic points $x,y \in \cO_{pt}(H;\cD),$ each pair of their cappings $\bar{x},\bar{y} \in \til{\cO}_{pt}(H;\cD)$ satisfies $\cA_{H,\cD}(\bar{x}) - \cA_{H,\cD}(\bar{y}) \notin A \cdot \Z,$ where $A = A_M = \omega([\C P^1]).$ Hence, as $c(H,a^j)$ is attained at a homogeneous cycle of degree $2(n-j),$ and $\nu(\lambda_j) \in A \cdot \Z,$ we obtain that $\{ c(\lambda_j a^j, H;\cD) \}_{0 \leq j \leq n}$ are all distinct. Hence by property \eqref{eq:max property filtration} of non-Archimedean filtrations, \[c(w, H; \cD) = \max \{c(a^j, H; \cD) - \nu(\lambda_j)\}.\] By continuity of spectral invariants, we obtain the analogous inequality for all non-degenerate $H,$ \[c(w, H) = \max \{c(a^j, H) - \nu(\lambda_j)\},\] and hence, again by continuity, for all $H \in \cH.$ Finally, for other choices of $\Lambda$ it is sufficient to use \cite[Sections 2.5, 6.1]{UsherZhang}, to observe that the same identity holds still under coefficient extension.
\end{proof}


\begin{cor}\label{cor: orthonormal}
For the choice of coefficients $\Lambda = \Lambda_{\tuniv},$ the ordered set $\{T^{c(a^j,H)} a^j \}_{0 \leq j \leq n}$ forms an orthonormal basis of $(QH_*(\C P^n, \Lambda),l_H).$
\end{cor}
	
\begin{lma}\label{lma: gamma vs mult beta_1 on CP^n}
Let $H \in \cH$ be a Hamiltonian with $\phi^1_H = \phi \in \Ham(\C P^n, \om_{FS}),$ then \[\beta_1([pt],H) = \gamma(\phi).\] 
\end{lma}	


\begin{proof}[Proof of Lemma \ref{lma: gamma vs mult beta_1 on CP^n}:] By Corollary \ref{cor: orthonormal}, it is enough to consider the matrix of $m_{[pt]}: QH_*(\C P^n, \Lambda) \to QH_*(\C P^n, \Lambda)$ in the orthonormal basis $\{T^{c(a^j,H)} a^j \}.$ Recalling that $[pt] = a^n,$ and $a^{n+1} = T^{A} a^0,$ for $A = A_M = \omega([\C P^1]),$ after changing the order of elements in the basis of the target, this matrix is diagonal of the form \[\Delta = \diag(T^{c(a^0,H) - c(a^{n}, H)}, \ldots, T^{c(a^n,H) - c(a^{n+n}, H)}).\] It is now sufficient, by \cite[Equation 26]{EntovPolterovichCalabiQM}, to note that the Seidel elements $S_{\eta,\sigma}$  for $\eta \in \pi_1 (\Ham(\C P^n))$ and $\sigma$ a section class, are all of the form $a^j T^{C_{\eta,\sigma}},$ for a certain exponent $C_{\eta,\sigma} \in \R$ (see \cite[Proposition 4.2]{EntovPolterovichCalabiQM}) thefore for each Hamiltonian $F \in \cH$ with $\phi^1_F = \phi,$ the spectral pseudonorm $\gamma([F]) = c([M],[F]) + c([M], [F]^{-1}) = c([M],[F]) - c([pt], [F])$ belongs to the set of exponents of the entries of $\Delta,$ and each such exponent can be written this way. Hence $\gamma(\phi) = \beta_1([pt],H),$ as minima of the same finite set.  
\end{proof}



Moving on to the proof itself, set $M = (\CP).$ By Sections \ref{subsec:relations}, \ref{subsec:QH}, \ref{subsec:spec} we now rewrite $\beta_1([pt],H)$ in the Lagrangian fashion $\beta_1([pt],\hat{H},L),$  where $L = \Delta_M \subset M \times M^{-},$ and $\hat{H} \in \cH_{M \times M^{-}}.$ Now for the Zoll metric on $M,$ and its induced metric on $M \times M^{-}$ and $L,$ $d_{C^0}(\id,\phi) < s_* = \frac{1-c}{1+c}$ implies $s = d_{C^0}(L, L') < s_*,$ where $L' = (\phi \times \id) L,$ whence $L' \subset s \cdot U$ for the Weinstein neighborhood $U = M\setminus \Sigma \cong D^* L.$ Here it is convenient to take $M = (\CP) \times (\C P^n,{-}\om_{FS}),$ while the normalization coming from the Zoll construction would give the twice smaller symplectic form $(\C P^n,\frac{1}{2} \om_{FS}) \times (\C P^n,-\frac{1}{2} \om_{FS}).$ Theorem \ref{thm:Viterbo-sharper}, \eqref{eq:high bar-length spectrum out and in}, and Lemma \ref{lma: gamma vs mult beta_1 on CP^n} imply \[\gamma(\phi) = \beta_1([pt],\hat{H},L) = \beta_1([pt],\ul L, \hat{H}^* \ul L; M) = \beta_1([pt],\ul L, \hat{H}^* \ul L; U) = \gamma(L, \phi_{\hat{H}}^{-1}L; U) \leq s \cdot \frac{c(1+c)}{1-c}.\]
%
This finishes the proof, since if $d_{C^0}(\id,\phi) \geq s_*,$ we obtain $\frac{c(1+c)}{1-c} \cdot d_{C^0}(\id,\phi) \geq c,$ however by \cite[Theorem G]{KS-bounds} we have $\gamma(\phi) \leq c,$ for all $\phi \in \Ham(\CP).$ 

\end{proof}


We note that a very similar argument allows to prove the following statement.


\begin{thm}\label{thm:C^0 rel}
Let $L \in \cl V$ be embedded in a symplectic manifold $M$ as a Lagrangian submanifold that is either weakly exact or wide monotone with $N_L > \dim L$ and $A_L = 1/2.$ Let $r \cdot D^*L$ for $r< s_* = \frac{1-c_L}{1+c_L}$ be embedded symplectically into $M$ as a Weinstein neighborhood of $L.$
Let $L'$ be Hamiltonian isotopic to $L$ in $M.$ If $L' \subset r \cdot D^*L,$  \[\gamma(L',L) \leq r\cdot \frac{c_L(1+c_L)}{2(1-c_L)},\] where $\gamma(L',L)$ is computed inside $M.$
\end{thm}


\begin{rmk}
An example of such an embedding is given by the corresponding $M \in \cl W.$ Theorem \ref{thm:C^0 rel} is a $C^0$-continuity statement for the Lagrangian spectral norm in a closed symplectic manifold. Statements of this kind were hitherto unknown. 
\end{rmk}

The proof, as the one above, rests on the following two lemmas.

\begin{lma}\label{lma:rel orthogonal basis}
		Let $L \in \cl V,$ $H \in \cH_M.$ There exists $a \in \displaystyle QH_{{n_L} - k}(L)$ where $n_L = \dim L,$ and $k|n_L$ such that $\{[L] = a^0,\ldots,a^n\}$ for $n = \frac{n_L}{k}$ is an orthogonal $\Lambda$-basis of $(QH_*(L),l_H).$
	\end{lma}

	We remark that the class $a$ is given by $a = [\R P^{n-1}], [\C P^{n-1}],[\bH P^{n-1}],[pt]$ for $L = \R P^n, \C P^n, \bH P^n, S^n$ respectively.
	
	\begin{lma}\label{lma:rel beta_1 and gamma}
		Let $L \in \cl V,$ and $L' = \phi^1_H L$ for $H \in \cH_{M}.$ Then \[\beta_1([pt],\LH) = \beta_1([pt],\ul L,{H^* \ul L})= \gamma(L',L)\] for each $\ul L$ decorating $L.$
	\end{lma}
	
 The proof of Lemma \ref{lma:rel orthogonal basis} is similar to that of Lemma \ref{lma: orthogonal basis}, and the proof of Lemma \ref{lma:rel beta_1 and gamma} is similar to that of Lemma \ref{lma: gamma vs mult beta_1 on CP^n}, with the only modification being running the argument of \cite[Proposition 4.2]{EntovPolterovichCalabiQM} for relative Seidel invariants \cite{huLalonde,huLalondeLeclercq,hyvrier,KS-bounds}.

\begin{proof}[Proof of Theorem \ref{thm:LSV}] 
Let $\phi = \phi^1_H.$ By \cite{LSV-conj}, it is enough to prove that the number of endpoints of $\cB_0(H),$ as defined with coefficients in $\Lambda_{\tmon},$ is equal to the number $X$ of endpoints of $\cB_{\Z}(H) = \sqcup_{r \in \Z} \cB_r([H]),$ as defined with coefficients in $\Lambda_{M,\tmin},$ that lie in the interval 
$[0,A_M),$ for $A_M = 2\kappa N_M.$ It is easy to see that 
\[\cB_0(H) = \sqcup_{0 \leq k < 2N_M} \cB_k(H;\Lambda_{M,\tmin})[-k\cdot\kappa].\]

Now, denoting for $\cB \in \barc$ by $X(\cB)$ the number of its endpoints, and $X^{+}(\cB),$ $X^{-}(\cB),$ the number of its upper, respectively, lower endpoints, we have \[X(\cB_0(H)) = \sum_{0 \leq k < 2N_M} X(\cB_k(H)).\] In other words $X(\cB_0(H))$ is the number of endpoints of all bars in $\cB_{\Z}([H])$ whose index is in $[0,2N_M).$ It remains to note that multiplication by the quantum variable $q \in \Lambda_{M,\tmin}$ gives an isomorphism $\cB_{r}(H) \cong \cB_{r-2N_M} (H)[-A_M].$ This means that the group $\Z$ acts on the multi-set of bars in $\cB_{\Z}(H).$ The quotient is easily identified with the multi-set of bars in $\cB_0(H).$ It is also easily identified with the multi-set $\cB_{-}$ of those bars in $\cB_{\Z}(H)$ whose lower end lies inside $[0,A_M).$ In particular, we have $X(\cB_0(H)) = X(\cB_{-}) = X^{+}(\cB_{-}) + X^{-}(\cB_{-}).$ Consider the multi-set $\cB_{+}$ of those bars in $\cB_{\Z}(H)$ whose uppper end lies inside $[0,A_M).$ Note that $\cB_{+}$ consists entirely of finite bars, hence $X^{+}(\cB_{+}) = X^{-}(\cB_{+}).$ By definition $X = X^{+}(\cB_{+}) + X^{-}(\cB_{-}),$ hence $X = X^{-}(\cB_{+}) + X^{-}(\cB_{-}).$ It is therefore enough to show that $X^{-}(\cB_{+}) = X^{+}(\cB_{-}).$ Since each bar is either infinite, in which case it contributes to neither $X^{+}(\cB_{-})$ nor $X^{-}(\cB_{+}),$ or finite, in which case it has both an upper and a lower end, it is enough to show that $X^{+}(\cB_{-} \setminus \cB_{+}) = X^{-}(\cB_{+} \setminus \cB_{-}).$ This identity, however, is evident, as both $\cB_{-} \setminus \cB_{+}$ and $\cB_{+} \setminus \cB_{-}$ are isomorphic to the quotient by the $\Z$-action of the invariant sub-multiset $\cB_{\Z}([H])^{\mrm{fin},\geq A_M} = \sqcup_{r \in \Z} \cB_r([H])^{\mrm{fin},\geq A_M} \subset \cB_{\Z}([H])$ consisting of finite bars whose length is at least $A_M.$

%

\end{proof}


\begin{proof}[Proof of Corollary \ref{cor:qm intr}.]
Theorem \ref{thm:Viterbo} implying that $\sigma = c([L],-)$ is a quasi-morphism follows by the same argument as the one for \cite[Proposition 3.5.3]{PolterovichRosen}, to wit: by the triangle inequality for spectral invariants
 \[c([L],F \#G) \leq c([L],F) + c([L],G),\] 
\[c([L],F) \leq c([L],F\#G) + c([L],\overline{G}),\] 
so that \[c([L],F) + c([L],G) - \gamma(L,[G]) \leq c([L],F \# G),\] hence Theorem \ref{thm:Viterbo} finishes the proof. Inequality \eqref{eq:zeta pb} is proven in the same way as \cite[Theorem 1.4]{EntPolZap-poisson}. The fact that $\mu$ is unbounded and vanishes on diffeomorphisms with displaceable support was proven in \cite[Theorem 1.3]{MonznerVicheryZapolsky}.
\end{proof}